\documentclass[fullpage,11pt]{article}

\usepackage[utf8]{inputenc}
\usepackage{amsmath,amsthm,amssymb,enumerate,framed,graphicx,array}

\newcommand{\Q}{\mathbb Q}
\newcommand{\N}{\mathbb{N}}
\newcommand{\Z}{\mathbb{Z}}
\newcommand{\R}{\mathbb{R}}
\newcommand{\conv}{\operatorname{conv}}
\newcommand{\setcond}[2]{\left\{ #1 \,:\, #2 \right\}}
\newcommand{\EndMarker}{\hspace*{\fill}~$\square$} 

\newcommand{\old}[1]{{}}

\newcommand{\intr}{\operatorname{int}}
\newcommand{\relintr}{\operatorname{relint}}

\newcommand{\dist}{\operatorname{dist}}
\newcommand{\cC}{\mathcal{C}}
\newcommand{\cK}{\mathcal{K}}
\newcommand{\bB}{\mathbb{B}}
\newcommand{\dotvar}{\,\cdot\,}
\newcommand{\eps}{\varepsilon}
\newcommand{\copr}{\mathbin{\Diamond}}

\newcommand{\cF}{\mathcal{F}}

\newcommand{\vol}{\operatorname{vol}}

\newtheoremstyle{mythmstyle}
	{\topsep}
	{\topsep}
	{\itshape}
	{}
	{\scshape}
	{.}
	{3pt}
	{}
\theoremstyle{mythmstyle}

\newtheorem{nn}{}[section]
\newtheorem{lemma}[nn]{Lemma}
\newtheorem{theorem}[nn]{Theorem}
\newtheorem{cor}[nn]{Corollary}
\newtheorem{prop}[nn]{Proposition}

\newtheorem{example}{Example}

\theoremstyle{definition}
\newtheorem{REMARK}[nn]{Remark}
\newenvironment{remark}{\begin{REMARK}}{\EndMarker\end{REMARK}}

\newcommand{\sprod}[2]{#1 \cdot #2 }
\newcommand{\aff}{\operatorname{aff}}
\newcommand{\thmheader}[1]{{\upshape (#1.)}}
\newcommand{\dd}{\operatorname{d}}

\numberwithin{equation}{section}

\title{Lifting properties of maximal lattice-free polyhedra}
\author{Gennadiy Averkov\footnote{Institute of Mathematical Optimization, Faculty of Mathematics, University of Magdeburg, Germany} \ and Amitabh Basu\footnote{Department of Applied Mathematics and Statistics, The Johns Hopkins University, MD, USA}}

\begin{document}

\maketitle

\begin{abstract} 
	We study the uniqueness of minimal liftings of cut-generating functions obtained from maximal lattice-free polyhedra. We prove a basic invariance property of unique minimal liftings for general maximal lattice-free polyhedra. This generalizes a previous result by Basu, Cornu\'ejols and K\"oppe~\cite{bcm} for {\em simplicial} maximal lattice-free polytopes, thus completely settling this fundamental question about lifting for maximal lattice-free polyhedra. We further give a very general iterative construction to get maximal lattice-free polyhedra with the unique-lifting property in arbitrary dimensions. This single construction not only obtains all previously known polyhedra with the unique-lifting property, but goes further and vastly expands the known list of such polyhedra. %Along the way, we develop tools which clarify topological properties of the space of maximal lattice-free polyhedra with unique-liftings which we believe to be of independent interest. 
Finally, we extend characterizations from~\cite{bcm} about lifting with respect to maximal lattice-free simplices to more general polytopes. These nontrivial generalizations rely on a number of results from discrete geometry, including the Venkov-Alexandrov-McMullen theorem on translative tilings and characterizations of zonotopes in terms of central symmetry of their faces.
\end{abstract}

\section{Introduction}

%\begin{framed}
%\textbf{Buffer:}
%
%The idea of \emph{cut-generating functions} goes back to the work of Gomory and Johnson from the 1970s \cite{MR0479415,MR0479416}. This idea has emerged as a major theme in the research of the last 5--6 years on cutting planes for mixed-integer linear programming. For an overview see, for example, the introduction in \cite{bcccz}. Note that whenever we use the term mixed-integer linear programming we assume integer linear programming to be included as a special case. Cut-generations functions originate in the study of \emph{mixed-integer sets} ****
%
%    The system defining $X_f(R,P)$ is a relaxation of an underlying mixed-integer linear problem, and yet is a powerful, unifying framework for studying general purpose cutting planes. 
%\end{framed}
%
\paragraph{Mixed-integer corner polyhedra.} 

%Let $d \in \N$ and $I$ be a fixed subset of $\{1, \ldots, d\}$. A linear mixed-integer optimization problem is the form
%\begin{equation}
%  \label{eq:orig-ip}
%  \setcond{c \cdot x}{A x =  b, x \in \R^d_+,  x_i \in\Z \quad \forall i\in I \,}
%\end{equation}
%
%For any subset $S \subseteq \{1, \ldots, d\}$, we will denote by $A_S$ and $x_S$ the corresponding columns and components of $A$ and $x$ respectively. \eqref{eq:orig-ip} is first solved by ignoring the integrality constraints using the simplex algorithm. This leads to a simplex tableau reformulation of the feasible region:
%\begin{equation}\label{eq:simplex-tab}
%   x_B = A_B^{-1}   b\, -\,   A_B^{-1} A_N x_N, \quad   x_{B\cap I}
%  \in \Z_+^{B\cap I},\   x_{B \setminus I} \in \R_+^{B \setminus I},\   x_{N\cap I} \in \Z_+^{N \cap I},\   x_{N \setminus I} \in \R_+^{N \setminus I}
%\end{equation}
%where $B$ and $N$ denote the basic and non-basic subsets of variables. To study the convex hull of solutions to~\eqref{eq:simplex-tab}, Gomory suggested dropping the non negativity constraints on the basic variables; the relaxation is known as the (mixed-integer) {\em corner polyhedron}. Notice that this means we may drop the constraints in~\eqref{eq:simplex-tab} corresponding to $x_{B\setminus I}$. 
%
The mixed-integer {\em corner polyhedron} is the convex hull of a {\em mixed-integer set} of the following form:
\begin{equation}
	\label{def mixed-int set}
	X_{f}(R,P) := \setcond{(s,y) \in \R_+^k \times \Z_+^\ell }{ f + Rs + Py\in \Z^n },
\end{equation}
%where $k = |N \setminus I|\, \ell = |N \cap I|, n = |B \cap I|$, $R \in \R^{n \times k}$ corresponds to the matrix $A^{-1}_BA_{N\setminus I}$, $s$ corresponds to the variables $x_{N\setminus I}$, $P  \in \R^{n \times \ell}$ corresponds to the matrix $A^{-1}_BA_{N\cap I}$, $y$ corresponds to the variables $x_{N\cap I}$ and $f = A^{-1}_Bb \R^n \setminus \Z^n$. 
%
where $k, \ell \in \Z_+, n \in \N$, $R \in \R^{n \times k}, \ P  \in \R^{n \times \ell}$ and $f \in \R^n \setminus \Z^n$. The set $X_f(R,P)$ was first studied by Gomory~\cite{MR0256718} for the purposes of generating cutting planes for general mixed-integer linear problems (MILPs). A short description of Gomory's idea is as follows. Consider a general MILP with a feasible region given in the standard form: 
	\begin{equation}
		\label{fis set gen MILP}
		\setcond{x \in \R_+^h \times \Z_+^q}{A x =b }.
	\end{equation}
	Here $A x = b$ is the defining linear system, $h$ is the number of continous variables and $q$ is the number of integer variables. The simplex method applied to the linear relaxation of the MILP decomposes the variables of $x$ into basic and non-basic ones. As a result, we can split $x$ into four vectors:
% continous non-basic, continous basic, integral non-basic and integral basic ones. Accordingly, we decompose $x \in \R_+^h \times \Z_+^q$ into four vectors: 
the vector $s \in \R_+^k$ of non-basic continuous variables, the vector $t \in \R_+^m$ of basic continuous variables, the vector $y \in \Z_+^\ell$ of non-basic integer variables and the vector $z \in \Z_+^n$ of  basic integer variables. 
	
	The basic variables can be expressed through the non-basic ones. That is, one has $z = f+ R s + P y$ and $t = g + U s + V y$ for appropriate matrices $U,V,R, P$ and appropriate vectors $f \in \R^n, g \in \R^m$. Gomory suggests relaxing the MILP by discarding the nonnegativity conditions on basic variables, that is, the conditions $t \in \R_+^m$ and $z \in \Z_+^n$ are relaxed to $t \in \R^m$ and $z \in \Z^n$. After this, $z$ and $t$ can eliminated from the problem description, since they are expressed through the non-basic variables; the condition $z \in \Z^n$ can also be reformulated without any use of $z$ as $f + R s + R y \in \Z^n$. This gives rise to the mixed-integer set $X_f(R,P)$ as defined in~\eqref{def mixed-int set}. %The convex hull of $X_f(R,P)$ is called the \emph{corner polyhedron}. 
Previous studies show that the corner polyhedron (the convex hull of $X_f(R,P)$) has a quite special facial structure, in sharp contrast to the facial structure of the convex hull of \eqref{fis set gen MILP}, which has much less structure in general. %On the other hand, mixed-integer sets can be associated to any linear mixed-integer problem (primarily, for the purpose of generating cutting planes). This explains why results on the corner polyhedron are useful for developing general-purpose solution methods in mixed-integer linear programming.
There has been a vast amount of literature, specially in the last 5-6 years, on utilizing the corner polyhedron for developing general-purpose solution methods in mixed-integer linear programming. We refer the reader to the survey~\cite{corner_survey} for this line of research.

	\paragraph{Cut generating functions.} In the 1970s Gomory and Johnson \cite{MR0479415,MR0479416,johnson1974group} suggested the following approach for generation of cuts for $X_f(R,P)$. We denote the columns of matrices $R$ and $P$ by $r_1,\ldots,r_k$ and $p_1,\ldots,p_{\ell},$ respectively. We allow the possibility that $k =0$ or $\ell = 0$ (but not both). Given $n \in \N$ and $f \in \R^n \setminus \Z^n$, a \emph{cut-generating pair} $(\psi, \pi)$ for $f$ is a pair of functions $\psi, \pi:\R^n \to \R$ such that 
\begin{equation}
	\label{psi pi ineq}
	\sum_{i=1}^k\psi(r_i)s_i + \sum_{j=1}^\ell\pi(p_j)y_j \ge 1
\end{equation}
is a valid inequality (also called a \emph{cut}) for the set $X_f(R,P)$ for every choice of $k, \ell \in \Z_+$ and for all matrices $R \in \R^ {n \times k}$ and $P \in \R^ {n \times \ell}$. We emphasize that cut-generating pairs depend on $n$ and $f$ and do \emph{not} depend on $k,\ell$, $R$ and $P$. For technical reasons, it is customary to concentrate on {\em nonnegative} cut-generating functions. This paper will also consider only nonnegative cut-generating pairs, i.e., $\psi \geq 0$ and $\pi \geq 0$. %The notion of generating-pair arose in the work of Gomory and Johnson REFERENCE TO WHICH SOURCE?.

%\begin{center}
%\begin{tabular}{|p{.98\textwidth}|}
%\hline
%\begin{framed}
\begin{example}\label{ex:GMI}
Let $n=1$ and $f \in \R\setminus \Z$. Define \begin{equation}\label{eq:GMI}\psi(r) = \max\bigg\{\frac{r}{1-[f]}, -\frac{r}{[f]}\bigg\} \qquad \pi(p) = \min\bigg\{\frac{[p]}{1-[f]}, \frac{1-[p]}{[f]}\bigg\},\end{equation} where $[x] = x - \lfloor x \rfloor$ denotes the fractional part of any real number $x$. Then $(\psi, \pi)$ forms a cut-generating pair for $f$; i.e., $\sum_{i=1}^k\psi(r_i)s_i + \sum_{j=1}^\ell\pi(p_j)y_j \ge 1$ is a valid inequality for $X_f(R,P)$. %This implies it is valid for $\conv\big(\setcond{(s,y, x) \in \R_+^k \times \Z_+^\ell\times \Z }{ \sum_{i=1}^k r_is_i + \sum_{j=1}^\ell p_jy_j + f = x }\big)$ for every $k, \ell \in \Z_+$ and real numbers $r_1, \ldots, r_k, p_1, \ldots, p_\ell$. Thus, $X_f(R,P)$ corresponds to a single row of a simplex tableaux of a linear mixed-integer problem.
In this case $X_f(R,P)$ and the pair $(\psi,\pi)$ are determined from a single row of the
simplex tableaux of the underlying MILP. \eqref{eq:GMI} gives the formula for the well-known {\em Gomory Mixed-Integer (GMI) cut}~\cite{gomory1960algorithm}. 
\end{example}
%\end{framed}
%\vspace{-0.1in}\\
%\hline
%\end{tabular}
%\end{center}

%Many important procedures for generation of cut-generating pairs (or more concretely, cuts for $X_f(R,P)$) are based on lattice-free sets. 
We call a subset $B$ of $\R^n$ \emph{lattice-free} if $B$ is $n$-dimensional, closed, convex and the interior of $B$ does not contain integer points. If $B$ is a lattice-free set and $f \in \intr(B)$, then $B$ can be defined analytically using the \emph{gauge function} of $B-f$, which is the function $\phi_{B-f}(r) := \inf \setcond{\rho > 0}{\frac{r}{\rho} \in B-f}.$ Intuitively, $\phi_{B-f}(r)$ plays the role of the length of $r$. Note that $\phi_{B-f}$ satisfies all the properties of the seminorm with the exception of the symmetry $\phi_{B-f}(r) = \phi_{B-f}(-r)$, which need not be fulfilled. Thus, $\phi_{B-f}$ induces an oriented ``distance'' on $\R^n$, where under orientation we mean that the ``distance'' from $a \in \R^n$ to $b \in \R^n$ need not be equal to the distance from $b$ to $a$. By the choice of $B$, the ``distance'' from $f$ to every point of $\Z^n$ is at least one. It was observed by Johnson~\cite{johnson1974group} that if $(\psi, \pi)$ is a cut-generating pair for $f$, then $\psi$ is the gauge function of $B-f$ for some lattice-free set $B$. Therefore, one approach to obtaining cut-generating pairs is to start with some lattice-free set $B$ with $f \in \intr(B)$, let $\psi$ be the gauge function of $B-f$ and find functions $\pi$ that can be combined with $\psi$ to form a valid cut-generating pair for $f$. For example, it is not hard to see that for any lattice-free set $B$ with $f \in \intr(B)$, $(\phi_{B-f}, \phi_{B-f})$ is a cut-generating pair. Indeed, notice that $\phi_{B-f}$ shares the following properties of a distance function: {\em positive homogeneity}, i.e., $\phi_{B-f}(\lambda r) = \lambda\phi_{B-f}(r)$ for every $r \in \R^n$ and $\lambda \geq 0$, and the {\em triangle inequality} or {\em subadditivity}, i.e., $\phi_{B-f}(r_1 + r_2) \leq \phi_{B-f}(r_1) + \phi_{B-f}(r_2)$ for every $r_1, r_2 \in \R^n$. Moreover, since $B$ is lattice-free, $\phi_{B-f}(x - f) \geq 1$ for every $x \in \Z^n$. So for any $(s,y) \in X_f(R,P)$, since $\sum_{i=1}^k r_is_i + \sum_{j=1}^\ell p_jy_j \in \Z^n - f$, we have $1 \leq \phi_{B-f}(\sum_{i=1}^k r_is_i + \sum_{j=1}^\ell p_jy_j) \leq \sum_{i=1}^k \phi_{B-f}(r_is_i) + \sum_{j=1}^\ell \phi_{B-f}(p_jy_j) =  \sum_{i=1}^k \phi_{B-f}(r_i)s_i + \sum_{j=1}^\ell \phi_{B-f}(p_j)y_j$.

In general, for a particular lattice-free set $B$ with $f\in \intr(B)$ and $\psi$ given by the gauge of $B-f$, there exist multiple functions $\pi$ that can be appended to make $(\psi, \pi)$ a cut-generating pair. If $(\psi,\pi)$ is a cut-generating pair, then $\pi$ is called a \emph{lifting} of $\psi$. %The function $\psi^*$ is often referred to as the {\em trivial lifting} %\footnote{The \emph{canonical lifting} of $\psi$ would be a more appropriate name for $\psi^\ast$, while it would be more legitimate to call $\psi$ itself the trivial lifting of $\psi$.} in the literature. 
The set of liftings of $\psi$ is partially ordered: we say that a lifting $\pi'$ for $\psi$ \emph{dominates} another lifting $\pi''$ for $\psi$ if $\pi'(r) \le \pi''(r)$ for every $r \in \R^n$. A \emph{minimal lifting} for $\psi$ is a lifting which is not dominated by another (distinct) lifting for $\psi$. A simple application of Zorn's lemma shows that that every lifting $\pi$ of $\psi$ is dominated by some minimal lifting $\pi'$ of $\psi$; see Theorem~1.1 and its proof in~\cite{basu2013k+1}.

\paragraph{Computations with cut-generating functions and unique minimal lifting.} The main idea behind cut-generating functions is to keep an arsenal of cut generating pairs $(\psi, \pi)$ that can be efficiently evaluated so that when we have a concrete MILP to solve, we ``plug in'' $r^1, \ldots, r^k$ into $\psi$ and $p^1, \ldots, p^\ell$ into $\pi$ and we obtain~\eqref{psi pi ineq} as a cutting plane for solving the MILP. Thus, we want $\psi$ and $\pi$ to be computable quickly, and thus require computational perspectives on the abstract notions of gauge and minimal liftings. This has been the focus of recent research on the corner polyhedron. We introduce these ideas next for motivating our work in this paper.

Lattice-free sets maximal with respect to inclusion are called \emph{maximal lattice-free}. It is known that every maximal lattice-free set in $\R^n$ is a polyhedron and the recession cone of $B$ is a linear space spanned by rational vectors; see \cite{MR1114315}, \cite{MR2724071} and \cite{MR3027668}. Then $B$ can be given by finitely many linear inequalities in the form 
\begin{equation}
	\label{B analytically}
	B = \setcond{x \in \R^n}{a_i \cdot (x-f) \le 1 \ \forall i \in I},
\end{equation}
where $I$ is a nonempty finite index set with at most $2^n$ elements (the fact that $|I| \leq 2^n$ is a theorem due to Lovasz~\cite{MR1114315} - see also~\cite{doignon} and~\cite{scarf}). An important observation in recent work is that, using the fact that the recession cone of $B$ is a linear space, the gauge function $\phi_{B-f}$ is computable by means of the simple formula
\begin{equation}\label{eq:formula-for-gauge}
	\phi_{B-f}(r) = \max \setcond{ \sprod{a_i}{r}}{i \in I}.
\end{equation}
%To explain the importance of lattice-free sets for generation of cuts for $X_f(P,R)$ we first consider the set $X_f(R,P)$ in the special case $\ell=0$. In this case $X_f(R,P) = \setcond{s \in \R_+^k}{ f + R s \in \Z^n}.$
%It can be seen directly that every linear nonstrict inequality valid for such $X_f(P,R)$ is either valid for the whole orthant $\R_+^k$ or otherwise can be written in the form 
%\begin{equation}
%	\label{psi ineq}
%	\sum_{i=1}^k \psi(r_i) s_i \ge 1,
%\end{equation}
%where $\psi$ is a gauge function of $B-f$ for some lattice-free polyhedron $B$ whose interior contains $f$. Enlarging $B$ we strengthen the respective inequality. Thus, the strongest inequalities of the form \eqref{psi ineq} are generated from \emph{maximal} lattice-free sets.
%

%, if both $\psi$ and $\pi$ coincide with the gauge function of $B-f$, then $(\psi,\pi)$ is a cut-generating pair. However, such a pair does not make any use of the information about the integrality of the variables $y=(y_1,\ldots,y_\ell)$ in \eqref{def mixed-int set}. Thus, potentially, there can be a way to improve such a pair to a stronger pair that generates even stronger cuts. Indeed, it is always to improve $\pi$ as follows. The condition $f+ R s + P y \in \Z^n$ remains unchanged if $P$ is replaced by $P+W$, where $W$ is an arbitrary $n \times k$ \emph{integral} matrix. This implies that $(\psi,\pi)$ for $\pi=\psi^\ast$ is a cut-generating pair, where $\psi^\ast$ is the following $\Z^n$-periodic function derived from $\psi$:

Given a maximal lattice-free set $B$, define the following $\Z^n$-periodic function derived from its gauge function $\phi_{B-f}$:

\begin{equation}
	\label{psi star def}
	\phi_{B-f}^\ast(r) := \inf_{w \in \Z^n} \phi_{B-f}(r+w).
\end{equation}

It is well-known that if $\psi$ is the gauge $\phi_{B-f}$ of a lattice-free set $B$ with $f \in \intr(B)$ and $\pi = \phi_{B-f}^*$, then $(\psi,\pi)$ is a cut-generating pair. To see this, observe that $X_f(R,P) = X_f(R, P+W)$ for any integral matrix $W$. %Thus, if $(\psi, \pi)$ is a cut-generating pair and for any $p\in \R^n$ and any $w \in \Z^n$, we define a new function $\pi'$ by changing the value of $\pi$ at $p$ and setting $\pi'(p) = \pi(p+w)$ (so $\pi'(r) = \pi(r)$ for all $r\neq p$), then $(\psi, \pi')$ is also a cut-generating pair. This implies that one can adjust all the values simultaneously by setting $\pi'(r) = \inf_{w \in \Z^n} \pi_{B-f}(r+w)$ and $(\psi, \pi')$ will also be a valid cut-generating pair. Recalling that $(\phi_{B-f}, \phi_{B-f})$ is a cut-generating pair, applying this idea shows that $(\phi_{B-f}, \phi^\ast_{B-f})$ is a valid cut-generating pair.
Consequently, the cut-generating pair $(\phi_{B-f}, \phi_{B-f})$ can be strengthened to $(\phi_{B-f}, \phi^\ast_{B-f})$.

%\begin{center}
%\begin{tabular}{|m{0.98\textwidth}|}
%\hline
\addtocounter{example}{-1}
\begin{example}[continued] \label{ex:GMI-2}
In Example~\ref{ex:GMI}, consider the closed interval with endpoints $\lfloor f \rfloor, \lceil f \rceil$ to be the lattice-free set $B$ (recall that we are working in $n=1$). Then, $\psi, \pi$ from the example are given by formulas~\eqref{eq:formula-for-gauge} and~\eqref{psi star def} respectively, as illustrated in this figure:
\end{example}

\begin{figure}[htbp]
\hspace{70pt}\scalebox{0.5}{\includegraphics{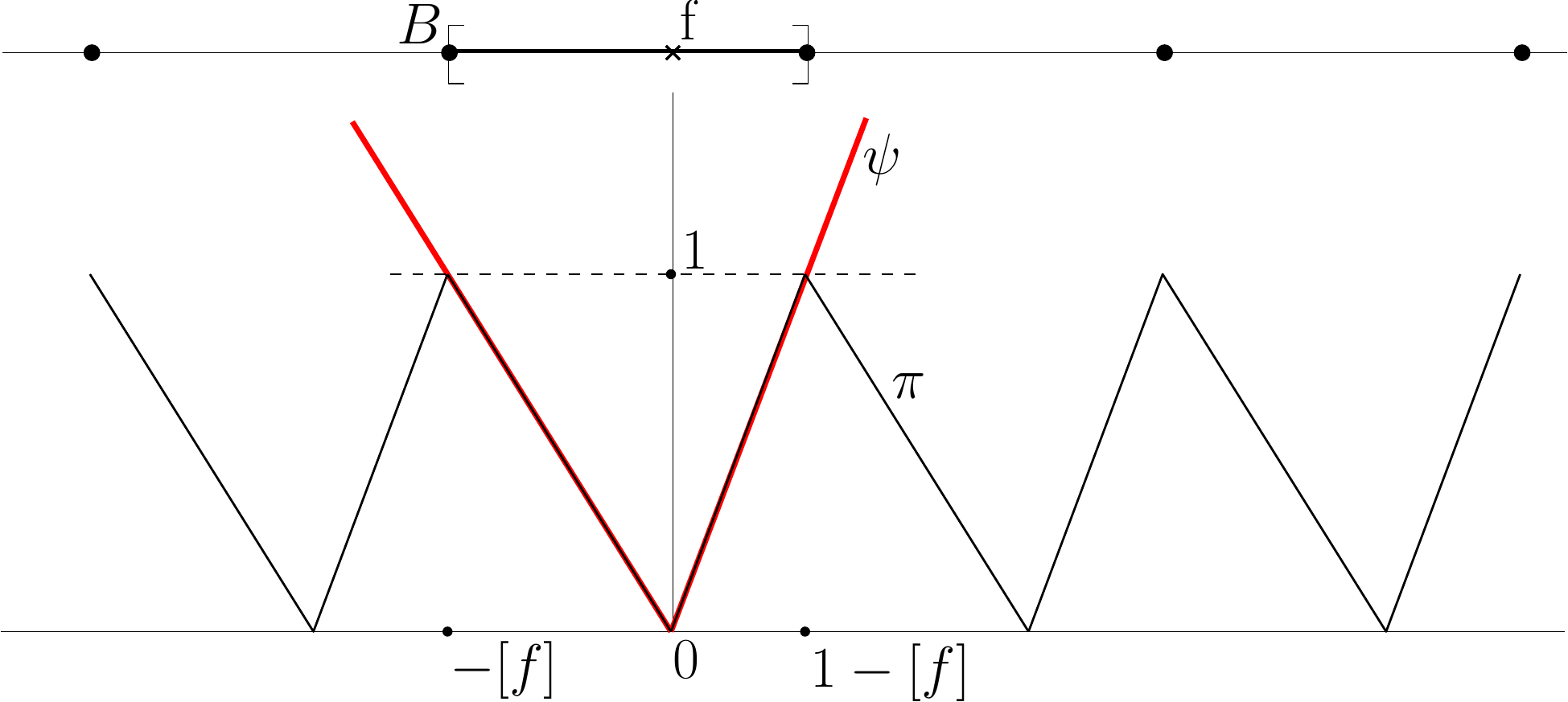}} \\
\end{figure}
%
%\hline
%\end{tabular}
%\end{center}

 Not all cut-generating pairs $(\psi, \pi)$ are of the form where $\psi$ is a gauge function of a maximal lattice-free convex set, and $\pi$ is a minimal lifting of $\psi$ - for some cut generating pairs $\psi$ may be the gauge of a lattice-free set that is not maximal. However, we have a nice formula for $\psi$ given by \eqref{eq:formula-for-gauge} when it is the gauge of a maximal lattice-free convex set; thus, one can compute very quickly the coefficients in a valid inequality. This is one reason to study minimal liftings of such special functions $\psi$. The hope is that minimal liftings may also often have closed form expressions that can be computed efficiently and thus the coefficients in \eqref{psi pi ineq} can be computed quickly. However, not many efficient procedures for computing minimal liftings are known in the literature. In fact, to the best of the authors' knowledge, \eqref{psi star def} is the only expression that has been put forward in the literature as a way to compute liftings. The following proposition makes precise the claim that formula \eqref{psi star def} is efficiently computable when $n$ is not part of the input (meaning, in practice, that $n$ is small). The proof of this proposition appears in the appendix.

\begin{prop}\label{thm:trivial-lifting-compute}
Let $n \in \N$, $f\in \R^n \setminus \Z^n$ and $B$ be a maximal lattice-free polyhedron given by \eqref{B analytically}, where $I$ is a nonempty finite index set. Then the infimum in \eqref{psi star def} is attained for some $w \in \Z^n$. Furthermore, for each fixed dimension $n \in \N$, the computational problem of determining $\phi_{B-f}^\ast(r)$ in \eqref{psi star def} from the input consisting of the rational vectors $a_i \in \Q^n$ with $i \in I$ and the point $f \in \Q^n \setminus \Z^n$, given in the standard binary encoding, is solvable in polynomial time.
\end{prop}

Recently, many authors have studied properties of minimal liftings; see \cite{bcccz,bcm,ccz,dw,dw2}. Given a maximal lattice-free polyhedron $B$ with $f \in \intr(B)$, there may exist multiple minimal liftings for the gauge function $\phi_{B-f}$, or there may be a unique minimal lifting. We say that a maximal lattice-free set $B$ has the \emph{unique-lifting property} with respect to $f \in \intr(B)$ if the gauge function of $B-f$ has exactly one minimal lifting. Otherwise, the gauge function has more than one minimal lifting and we say that $B$ has the \emph{multiple-lifting property} with respect to $f$. Maximal lattice-free sets with the unique-lifting property give rise to concrete formulas for cutting planes, because $\phi_{B-f}^\ast$ defined in~\eqref{psi star def} becomes the unique minimal lifting, and we can use the formula~\eqref{eq:formula-for-gauge} for $\phi_{B-f}$ and Proposition~\ref{thm:trivial-lifting-compute} to quickly compute the coefficients in~\eqref{psi pi ineq}.

Moreover, the function $\phi_{B-f}^\ast$ given by \eqref{psi star def} is not always a minimal lifting of $\phi_{B-f}$. The following proposition shows the importance of unique liftings in this context (the proof appears in the appendix).

\begin{prop}\label{thm:trivial-lifting-minimal}
Let $B$ be a maximal lattice-free polyhedron with $f\in \intr(B)$. Then $\phi_{B-f}^\ast$ defined by \eqref{psi star def} is a minimal lifting for the gauge function $\phi_{B-f}$ of $B-f$ if and only if $\phi_{B-f}$ has a unique minimal lifting. 
\end{prop}

We hope the above discussion lends credence to the claim that characterizing pairs $B, f$ with unique minimal liftings is an important question in the cut-generating function approach to cutting planes. The purpose of this manuscript is to study maximal lattice-free sets with the unique-lifting property. 

\paragraph{Our contributions.} We summarize our main contributions in this paper.
\begin{enumerate}[(i)]
	\item \emph{Invariance result.} A natural question arises: is it possible that $B$ has the unique-lifting property with respect to one $f_1 \in \intr(B)$, and has the multiple-lifting property with respect to another $f_2 \in \intr(B)$? This question was investigated in~\cite{bcm} and the main result was to establish that this cannot happen when $B$ is a {\em simplicial} polytope. We prove this for general maximal lattice-free polytopes without the simpliciality assumption: for every maximal lattice-free set $B$ either $B$ has the unique-lifting property with respect to every $f \in \intr(B)$ or otherwise $B$ has the multiple-lifting property with respect to every such $f$ (see Theorem~\ref{inv:thm}). 
	In view of this result, we can speak about the unique-lifting property of $B$, without reference to any $f \in \intr(B)$.

	\item \emph{Result on the volume of the lifting region modulo $\Z^n$.} To prove the mentioned invariance result, we study the so-called {\em lifting region} (defined precisely in Section~\ref{sec:basic-notions}), and show that its volume modulo the lattice $\Z^n$, is an affine function of $f$ (see Theorem~\ref{thm:affine}). This is also an extension of the corresponding theorem from~\cite{bcm} for simplicial $B$. Besides handling the general case, our proof is also significantly shorter and more elegant. We develop a tool for computing volumes modulo $\Z^n$, which enables us to circumvent a complicated inclusion-exclusion argument from~\cite{bcm} (see pages 349-350 in \cite{bcm}). %We view this volume computation tool as an important technical contribution of this paper.
	
	\item \emph{Topological result.} In Section~\ref{limits} we show that in the space of all maximal lattice-free sets, endowed with the Hausdorff metric, the subset of the sets having the unique-lifting property is closed (see Theorem~\ref{thm:limit}). This topological property turns out to be useful for verification of the unique-lifting property of maximal lattice-free sets built using the coproduct operations (see below).

	\item \emph{Constructions involving the coproduct operation.} Our techniques give an iterative procedure to construct new families of polytopes with the unique-lifting property in every dimension $n \in \N$. This vastly expands the known list of polytopes with the unique-lifting property. Furthermore, the coproduct operation enables to construct \emph{all} sets with the unique-lifting property for $n=2$. See Section~\ref{constr sect}; in particular, Theorem~\ref{coproduct thm} and Corollaries~\ref{coproduct of k sets cor},~\ref{coproduct lat free cor} and~\ref{cor:a1-an}.

	\item \emph{A characterization for special polytopes.}  A major contribution of~\cite{bcm} was to characterize the unique-lifting property for a special class of simplices. We generalize all the results from~\cite{bcm} to a broader class of polytopes called {\em pyramids}, which are constructed using the so-called coproduct operation; see Theorem~\ref{thm:unique-integer} and Theorem~\ref{thm:base-condition}. For these generalizations, we build tools in Section~\ref{spec pyr sect} that rely on nontrivial theorems from the geometry of numbers and discrete geometry, such as the Venkov-Alexandrov-McMullen theorem for translative tilings in $\R^n$ and McMullen's characterizations of polytopes with centrally symmetric faces~\cite{mcmullen}.% and the combinatorial structure of zonotopes.
\end{enumerate}
	
All results that are stated in this paper for general maximal lattice-free polytopes, also hold for maximal lattice-free polyhedra $B$ that are unbounded. This follows from the fact that for unbounded maximal lattice-free sets, the lifting region can be viewed as a cylinder over the lifting region of a lower dimensional maximal lattice-free polytope. Restricting to polytopes keeps the presentation less technical.

In our arguments we rely on tools from convex geometry, the theory of polytopes and the geometry of numbers; for the background information see \cite{gruber}, \cite{MR0274683}, \cite{MR1216521},  \cite{ziegler}. The necessary basic notation and terminology, used throughout the manuscript, is introduced in the beginning of the following section. 

\section{Preliminaries}\label{sec:basic-notions}

\paragraph{Basic notation and terminology.} 

Let $n \in \N$. We will also use $e_1, e_2, \ldots, e_n$ to denote the standard unit vectors of $\R^n$. The notation $\vol$ stands for the $n$-dimensional volume (i.e., the Lebesgue measure) in $\R^n$ with the standard normalization $\vol([0,1]^n)=1$. For $a, b \in \R^n$ we introduce the segment $[a,b]$ and the relatively open segment $(a,b)$ joining $a$ and $b$ by:
\begin{align*}
	[a,b] &:= \setcond{(1-\lambda) a + \lambda b }{0 \le \lambda \le 1},
	\\ (a,b) &:= \setcond{(1-\lambda) a + \lambda b}{0<\lambda < 1}.
\end{align*}

We will denote the convex hull, affine hull, interior of a set $X$ and the relative interior of a convex set $X$ by $\conv(X), \aff(X), \intr(X)$ and $\relintr(X)$, respectively. 

Given a polytope $P$, we denote by $\cF(P)$ the set of all faces of $P$. For an integer $i \ge -1$, by $\cF^i(P)$ we denote the set of all $i$-dimensional faces of $P$. Note that  $\cF^{-1}(P) = \{\emptyset\}$ and $\cF^n(P)= \{P\}$. Elements of $\cF^0(P)$ are called \emph{vertices} of $P$. Elements of $\cF^i(P)$ with $i=\dim(P)-1$ are called \emph{facets} of $P$.

We shall make use of the following types of polytopes. Let $P$ be a polytope in $\R^n$:
\begin{itemize}
	\item $P$ is called a \emph{pyramid} if $P$ can be represented by $P=\conv(P_0 \cup \{a\})$, where $P_0$ is a polytope and $a$ is a point lying outside $\aff(P_0)$. In this case the polytope $P_0$ is called the \emph{base} of $P$ and $a$ is called the \emph{apex} of $P_0$.
	\item $P$ is called a \emph{double pyramid} if $P$ can be represented by $P = \conv(P_0 \cup \{a_1,a_2\})$, where $P_0$ is a polytope and $a_1,a_2$ are points such that $[a_1,a_2]$ intersects $P_0$, but neither $a_1$ nor $a_2$ is in the affine hull of $P_0$.
	\item $P$ is called a \emph{spindle} if $P$ can be represented by $P= (P_1 + a_1) \cap (P_2 + a_2)$, where $P_1,P_2$ are pointed polyhedral cones and $a_1,a_2$ are points satisfying $a_i \in \relintr(P_j) + a_j$ for $\{i,j\} = \{1,2\}$. The points $a_1,a_2$ are called the \emph{apexes} of $P$.
\end{itemize}

\paragraph{Geometric characterization of unique minimal liftings.} The authors of \cite{bcccz} were able to characterize the unique lifting property in a purely geometric way. Let $B$ be a maximal lattice-free polytope in $\R^n$ and let $f \in B$ (not necessarily in the interior of $B$). With each $F\in \cF(B) \setminus \{\emptyset,B\}$ and $f$ we associate $\conv(\{f\} \cup F)$, which is a pyramid of dimension $\dim(F)+1$ whenever $f \not \in F$. For every $z \in F \cap \Z^n$ we define the polytope
\[
	S_{F,z}(f) := \conv( \{f\} \cup F) \cap \bigl(z+f - \conv(\{f\} \cup F) \bigr),
\]
given as the intersection of $\conv(\{f\} \cup F)$ and the reflection of $\conv(\{f\} \cup F)$ with respect to $(z+f)/2$. Note that, if $f \not\in F$ (which is a generic situation), then $S_{F,z}(f)$ is a spindle. Furthermore, we define 
\[
	R_{F}(f) := \bigcup_{z \in F \cap \Z^n} S_{F,z}(f),
\]
the union of all sets $S_{F,z}(f)$ arising from the face $F$. 

The set
\[
	R(B,f) := \bigcup_{ F \in \cF(B) \setminus \{\emptyset,B\}} R_F(f) 
\]
 is called the {\em lifting region} of $B$ associated with the point $f$. In \cite{bcccz} it was shown that for $f\in \intr(B)$
%\begin{equation}
%	\label{R and R'}
%	R(B,f) - f = R'(B,f)
%\end{equation}
%for every $f \in \intr(B)$ (this result justifies the name for $R(B,f)$). We emphasize that while $R'(B,f)$ was defined for the case $f \in \intr(B)$ only, the set $R(B,f)$ is also defined and will be used in the following considerations for the case of $f$ lying in the boundary of $B$. 
%
%
%Therefore, taking into account \eqref{R and R'} and the definition of $R'(B,f)$ (see \eqref{un lift char}), for $f \in \intr(B)$ we obtain:
\begin{align*}
	B \ & \text{has the unique-lifting property with respect to $f$} & &\Longleftrightarrow & R(B,f) + \Z^n &= \R^n.
\end{align*}

Thus, in the rest of the paper, we study the covering properties of $R(B,f)$ by lattice translates to analyze the unique-lifting property of $B$ with respect to $f\in \intr(B)$.

We observe that, since $R_{F_1}(f) \subseteq R_{F_2}(f)$ for $F_1,F_2 \in \cF(B) \setminus \{B, \emptyset\}$ satisfying $F_1 \subseteq F_2$, the lifting region $R(B,f)$ can also be represented using the set $\cF^{n-1}(B)$ of all facets of $B$ as follows:
\[
	R(B,f) = \bigcup_{F \in \cF^{n-1}(B)} R_F(f).
\]

\section{Invariance theorem on the uniqueness of lifting}

\paragraph{Integral formula for the volume of a region modulo $\Z^n$.} For $t \in \R$ let $[t] = t - \lfloor t \rfloor$ be the fractional part of $t$. For any set $X\subseteq \R^n$, define $X / \Z^n := \{[x]:  x \in X\} \subseteq [0,1]^n$. Observe that a compact set $X\subseteq\R^n$ covers $\R^n$ by lattice translations, i.e., $X+\Z^n = \R^n$ if and only if $\vol( X / \Z^n) = 1$.

Let $X$ be a finite subset of $\R^n$ and assume that we wish to count the number of elements in $X / \Z^n$.  Then it suffices to consider all $x \in X$ and count the element $x$ with the weight $1/|X \cap (x + \Z^n)|$, because all elements of $X \cap (x + \Z^n)$ also belong to $X$ and generate the same element $[x]$. The following lemma is a ``continuous counterpart'' of the above combinatorial observation.

\begin{lemma}\label{lem:torus-vol}
Let $R\subseteq \R^n$ be a compact set with nonempty interior. Then
\begin{equation}
	\label{volume of wrapped set}
	\vol( R / \Z^n ) = \int_R \frac{\dd x}{|R \cap (x + \Z^n)|}.
\end{equation}
\end{lemma}
\begin{proof}
	We use the following formula for substitution of integration variables in the case that the underlying substitution function $f : R \rightarrow \R^n$ is not necessarily injective: 
	\begin{equation}
		\label{subst integr eq}
		\int_{\R^n} \sum_{\Tilde{x} \in R \, : \, f(\Tilde{x}) = y} g(\Tilde{x}) \, \dd y = \int_{R} g(x) |\det(\nabla f)(x) | \, \dd x.
	\end{equation}
	Here $g : R \rightarrow \R^n$ is a Lebesgue measurable function, $f : R \rightarrow \R^n$ is an almost everywhere Lipschitz function and $\nabla f$ denotes the Jacobian matrix of $f$. Note that in the case of injective $f$, we get a well-known substitution formula with the left hand side equal to $\int_{f(R)} g(f^{-1}(y))\, \dd y$. Formula \eqref{subst integr eq} is a standard fact in geometric measure theory; it is a special case of Corollary~5.1.3 in \cite{MR2427002}. We use \eqref{subst integr eq} in the case $f(x) = ([x_1],\ldots,[x_n])$ and $g(x) = \frac{1}{|R \cap (x + \Z^n)|}$, where $x=(x_1,\ldots,x_n) \in R$. Clearly, $(\nabla f)(x)$ is the identity matrix for almost every $x \in R$. Thus, the right hand side of \eqref{subst integr eq} coincides with the right hand side of \eqref{volume of wrapped set}. We analyze the left hand side of \eqref{subst integr eq}. Consider an arbitrary $y \in f(R)$, that is, $y \in \R^n$ and $f(x) = y$ for some $x \in R$. We fix $x$ as above. For every $\Tilde{x} \in R$ the equality $f(\Tilde{x}) = y$ can be written as $f(\Tilde{x}) = f(x)$, which is equivalent to $x - \Tilde{x} \in \Z^n$. Consequently, $g(x) = g(\Tilde{x})$ and $\setcond{\Tilde{x} \in R}{ f(\Tilde{x}) = y} = R \cap (x + \Z^n)$. It follows that the sum on the left hand side of \eqref{subst integr eq} is equal to $1$ for every $y \in f(R)$, and $0$ for every $y \not\in f(R)$. Thus, the the left hand side of \eqref{subst integr eq} is equal to $\int_{f(R)} \dd y$ and, by this, coincides with the left hand side of \eqref{volume of wrapped set}.
\end{proof}

\paragraph{Structure of congruences modulo $\Z^n$ for points of the lifting region.} 

%With each $F \in \fct(F)$ we associate an affine function $a_F$ vanishing on $F$, positive on $\intr(B)$ and satisfying $a_F(\Z^n) = \Z$. Analytically, one can define $a_F$ by $a_F(x) := h_F - \sprod{u_F}{x}$, where $u_F \in \Z^n \setminus \{o\}$ is outer normal of $F$ satisfying $\gcd(u_F)=1$ and $h_F$ is the value of $\sprod{u_F}{x}$ for $x$ belonging to $F$.

\begin{lemma}
	\label{when collision?}
	Let $n \in \N$, let $B$ be a maximal lattice-free polytope in $\R^n$ and let $f \in B$.
	Let $F_1, F_2 \in \cF^{n-1}(B)$ and let $z_i \in \relintr(F_i) \cap \Z^n$ for $i \in \{1,2\}$. Suppose $x_1 \in \intr(S_{F_1, z_1}(f))$ and $x_2 \in \intr(S_{F_2, z_2}(f))$ be such that $x_1 - x_2 \in \Z^n$. Then $F_1 = F_2$ and, furthermore, the vector $x_1-x_2$ is the difference of two integral points in the relative interior of $F_i$ ($i \in \{1,2\}$), i.e.:
	\begin{equation}
		\label{struct congr:eq}
		x_1 - x_2 \in \relintr(F_i) \cap \Z^n - \relintr(F_i) \cap \Z^n.
	\end{equation}
	In particular, the vector $x_1-x_2$ is parallel to the hyperplane $\aff(F_i)$.
\end{lemma}

\begin{proof}
	%For $i \in \{1,2\}$ choose $F_i \in \fct(B)$ and  with $x_i \in S_{F_i,z_i}(f)$. 
For $i \in \{1,2\}$, if $f$, $x_i$ and $z_i$ do not lie on a common line, we introduce the two-dimensional affine space $A_i := \aff \{f,x_i,z_i\}$. Otherwise choose $A_i$ to be any two-dimensional affine space containing $f$, $x_i$ and $z_i$. The set $T_i:=\conv(F_i \cup \{f\}) \cap A_i$ is a triangle, whose one vertex is $f$. We denote the other two vertices by $a_i$ and $b_i$. Observe that $a_i, b_i$ are on the boundary of facet $F_i$ such that the open interval $(a_i, b_i) \subseteq \relintr(F_i)$. Since $z_i$ lies on the line segment connecting $a_i, b_i$ and $z_i \in \relintr(F_i)$, there exists $0 < \lambda_i <1$ such that $z_i = \lambda_i a_i + (1-\lambda_i)b_i$. Since $x_i \in \intr(S_{F_i, z_i}(f))$, there exist $0 < \mu_i, \alpha_i, \beta_i < 1$ such that $x_i = \mu_i f + \alpha_i a_i + \beta_i b_i$  and $\mu_i + \alpha_i + \beta_i = 1$. Also, observe that $x_i \in \relintr(T_i \cap (z_i + f - T_i))$. Therefore, $\alpha_i < \lambda_i$ and $\beta_i < 1 -\lambda_i$. 

Consider first the case $\mu_1 \geq \mu_2$. In this case,  $z_2+ x_1  - x_2$ is an integral point, which can be represented by
\begin{equation*}
	z_2 + x_1 - x_2 = (\mu_1 - \mu_2)f +  (\lambda_2 - \alpha_2 )a_2 + (1-\lambda_2 - \beta_2)b_2 + \alpha_1a_1 + \beta_1b_1
\end{equation*}

Observe that $(\mu_1 - \mu_2) +  (\lambda_2 - \alpha_2 )+ (1-\lambda_2 - \beta_2) + \alpha_1 + \beta_1= 1$, each of the terms in the sum is nonnegative, and the coefficients $\lambda_2 - \alpha_2 , 1-\lambda_2 - \beta_2, \alpha_1, \beta_1$ are strictly positive. Further, if $F_1 \neq F_2$, then for every facet $F \in \mathcal{F}^{n-1}(B)$, at least one of the points $a_1, b_1, a_2, b_2$ does not lie on the facet $F$ (because $(a_i, b_i) \subseteq \relintr(F_i)$); therefore, the point $z_2 + x_1 - x_2$ does not lie on any facet. This would imply that $z_2 + x_1 - x_2 \in \intr(B)$ contradicting the fact that $B$ is lattice-free. Therefore, $F_1 = F_2$. Now for every facet $F\neq F_1$ at least one of $a_1, b_1, a_2, b_2$ is again not on $F$, and so the point $z_2 + x_1 - x_2$ does not lie on $F$. Thus, since $\intr(B)$ is lattice-free, $z_2 + x_1 - x_2$ must lie on $F_1$; and further $z_2 + x_1 - x_2 \in \relintr(F_1)$. Since $F_1=F_2$, $z_2 \in \relintr(F_1)$ otherwise, $S_{F_2, z_2}(f)$ has empty interior. Thus, we obtain that $x_1 - x_2$ is the difference of two integral points in the relative interior of $F_1$.

% $z_2 + x_1 - x_2$ is a convex combination of points in $B$. Since $B$ has no point from $\Z^n$ in its interior, and $f \in \intr(B)$, the coefficient of $f$ in the above expression must be 0. Thus, $\mu_1 = \mu_2$ and, $z_2+x_1-x_2$ is represented by 
%\begin{equation}
%	\label{z2+x1-x2:eq}
%	z_2 + x_1 - x_2 =  (\lambda_2 - \alpha_2 )a_2 + (1-\lambda_2 - \beta_2)b_2 + \alpha_1a_1 + \beta_1b_1,
%\end{equation}
%where all coefficients are strictly positive. Since $(a_i, b_i) \subseteq \relintr(F_i)$, we cannot have $F_1 \neq F_2$, because otherwise \eqref{z2+x1-x2:eq} would imply $z_2 + x_1 - x_2 \in \intr(B)$, which contradicts the assumption that $B$ is lattice-free. It follows that $F_1 = F_2 =: F$. Furthermore, \eqref{z2+x1-x2:eq} yields that $z_2 + x_1 - x_2$ is a convex combination of a point of $(a_1,b_1)$ and a point of $(a_2,b_2)$, where both $(a_1,b_1)$ and $(a_2,b_2)$ are subsets of $\relintr(F)$. Consequently, $z_2 + x_1 -x_2$ belongs to $\relintr(F) \cap \Z^n$ and by this $x_1-x_2$ is a difference of two points of $\relintr(F) \cap \Z^n$.

The case $\mu_1 \leq \mu_2$ is similar with the same analysis performed on $z_1 + x_2 - x_1$.\end{proof}

\paragraph{Invariance theorem for unique liftings.} We now have all the tools to prove our main invariance result about unique minimal liftings. The main idea is to show that the volume of the lifting region modulo $\Z^n$ is the restriction of an affine function.

We first recall a basic fact about affine transformations.

\begin{lemma}
	\label{lin alg lem}
	Let $H$ be a hyperplane in $\R^n$ and let $f^\ast,f$ be two points in $\R^n \setminus H$ that lie in the same open halfspace determined by $H$. Let $T$ be the affine transformation on $\R^n$, given by $x \mapsto T(x) := A x + b$ with $A \in \R^{n \times n}$ and $b \in \R^n$, that acts identically on $H$ and sends $f^\ast$ to $f$. For every $x \in \R^n$ let $\delta(x)$ denote the Euclidean distance of $x$ to $H$. Then $\det(A) = \frac{\delta(f)}{\delta(f^\ast)}$.
\end{lemma}
\begin{proof}
	By changing coordinates using a rigid motion (i.e., applying an orthogonal transformation and a translation such that distances, angles and volumes are preserved), we can assume that $H= \R^{n-1} \times \{0\}$ and $f^\ast, f \in \R^{n-1} \times \R_{>0}$. For every $a \in \R^{n-1} \times \R_{>0}$, the value $\delta(a)$ is the last component of $a$. It is easy to see that the linear transformation $T_a$  that keeps $e_1,\ldots,e_{n-1}$ unchanged and sends $e_n$ onto $a$ is given by a matrix whose determinant is $\delta(a)$. Since $T = T_{f} T_{f^\ast}^{-1}$, we conclude that $\det(A) = \frac{\delta(f)}{\delta(f^\ast)}$.
\end{proof}

%Theorem~\ref{inv:thm} will follow as a direct consequence of the following result.

\begin{theorem} 
	\label{thm:affine}
Let $n \in \N$ and let $B$ be a maximal lattice-free polytope in $\R^n$. Then the function $f \mapsto \vol(R(B,f)/ \Z^n)$, acting from $B$ to $\R$, is the restriction of an affine function.
\end{theorem}

\begin{proof}%[Proof of Theorem~\ref{thm:affine}]

First, we observe that 
\begin{equation}
	\label{vol is sum of facet parts:eq}
	\vol(R(B,f) / \Z^n)  =  \sum_{F \in \cF^{n-1}(B)}\vol(R_F(B,f) / \Z^n)
\end{equation}
because, by Lemma~\ref{when collision?},  for two distinct facets $F_1$ and $F_2$ of $B$, no point of $\intr(R_{F_1}(B,f))$ is congruent to a point of $\intr(R_{F_2}(B,f))$ modulo $\Z^n$ . Therefore, it suffices to consider an arbitrary $F \in \cF^{n-1}(B)$ and show that the mapping $f \in B \mapsto V(f):=\vol(R_F(B,f) / \Z^n)$   is a restriction of an affine function. We fix a vector $f^\ast \in \intr(B)$. For every $f \in B$, let $\delta(f)$ be the Euclidean distance of $f$ to the hyperplane $\aff(F)$. It suffices to show that for every $f \in B$ one has 
	\begin{equation}
		\label{ratios eq}
		\frac{V(f)}{V(f^\ast)} = \frac{\delta(f)}{\delta(f^\ast)},
	\end{equation}
	because  \eqref{ratios eq} implies $V(f) = \frac{V(f^\ast)}{\delta(f^\ast)} \delta(f)$, where $\frac{V(f^\ast)}{\delta(f^\ast)}$ is a constant factor and $f \in B \mapsto \delta(f)$  is a restriction of an affine function. From now on, we fix $f \in B$. If $f \in F$, then \eqref{ratios eq} clearly holds, because both $V(f)$ and $\delta(f)$ vanish. Therefore, we assume $f \in B \setminus F$. By Theorem~\ref{lem:torus-vol}, we have the following integral expressions for $V(f^\ast)$ and $V(f)$: 
	\begin{align}
		\label{int expressions}
		V(f^\ast) & = \int_{R_F(f^\ast)} \frac{\dd x^\ast}{|R_F(f^\ast) \cap (x^\ast + \Z^n)|}, & V(f) & = \int_{R_F(f)} \frac{\dd x}{|R_F(f) \cap (x + \Z^n)|}.
	\end{align}
	Consider the bijective affine transformation $T$ which acts identically on $\aff (F)$ and sends $f^\ast$ to $f$. We want to relate both integral expressions in \eqref{int expressions} by changing the integration variables using the transformation $T$. We will derive \eqref{ratios eq} by substituting $T(x^ \ast)$ for $x$ in the integral expression for $V(f)$. For this purpose, it is sufficient to verify the following three conditions dealing with the domains of integration, the integrands and the determinant of the Jacobian matrix of $T$, respectively:
	\begin{enumerate}
		\item $T$ maps $R_F(f^\ast)$ onto $R_F(f)$, i.e., $T(R_F(f^\ast)) = R_F(f)$.
		\item One has $R_F(f) \cap (x +\Z^n) = R_F(f^ \ast) \cap (x^\ast + \Z^n)$ if $x = T(x^\ast)$ and $x^\ast \in R_F(f^\ast)$.
		\item One has $| \det \nabla T | = \frac{\delta(f)}{\delta(f^\ast)}$, where $\nabla T$ denotes the Jacobian matrix of $T$.
	\end{enumerate} 
	For verifying Condition~1 we recall that
	the set $R_F(f)$ is defined using the pyramids $\conv(F \cup \{f\})$ and the reflected pyramids $f + z - \conv(F \cup \{f\})$ with $z \in F \cap \Z^n$. Clearly, $T$ maps $\conv(F \cup \{f^\ast\})$ onto $\conv(F \cup \{f\})$. But then $T$ also maps $f^\ast + z - \conv(F \cup \{f^\ast\})$ with $z \in F \cap \Z^n$ onto $f + z - \conv(F \cup \{f\})$ because an element of $f^\ast + z - \conv(F \cup \{f^\ast\})$ is an affine combination of $f^\ast$, $z$ and a point of $\conv(F \cup \{f^\ast\})$ with coefficients $1$, $1$ and $-1$. Furthermore, from the definition of $T$ we get $T(f^ \ast) = f$, $T(z) = z$ and $T(\conv( F \cup \{f^\ast\}) = \conv(F \cup \{f\})$. Thus, $T( R_F(f^\ast)) = R_F(f))$ and Condition~1 is fulfilled.
	
	We verify Condition~2. Choose an arbitrary $y^\ast \in R_F(f^\ast) \cap (x^\ast +\Z^n)$. Then $x^\ast, y^\ast \in R_F(f^\ast)$, $x^\ast - y^\ast \in \Z^n$ and, by Lemma~\ref{when collision?}, the segment $[x^\ast,y^\ast]$ (which is possibly degenerate to a point) is parallel to $\aff(F)$. Since $T$ acts identically on $\aff(F)$ and $[x^\ast,y^\ast]$ is parallel to $\aff(F)$, the segment $[T(x^\ast),T(y^\ast)]$ is a translation of $[x^\ast,y^\ast]$. In particular, $T(y^\ast) - T(x^\ast) \in \Z^n$ and by this, taking into account Condition~1, we see that $T(y^\ast)$ is an element of $T(R_F(f^\ast)) \cap (T(x^\ast) + \Z^n) = R_F(f) \cap (x + \Z^n)$. This shows that the image of $R_F(f^ \ast) \cap (x^\ast + \Z^n)$ under $T$ is a subset of $R_F(f) \cap (x +\Z^n)$. Interchanging $x$ and $x^\ast$ and replacing $T$ by $T^{-1}$, the above argument can also be used to show that the image of $R_F(f) \cap (x +\Z^n)$ under $T^{-1}$ is a subset of $R_F(f^ \ast) \cap (x^\ast + \Z^n)$. The latter verifies Condition~2.
	
	Condition~3 is fulfilled in view of Lemma~\ref{lin alg lem}.
	
	The above shows \eqref{ratios eq} and yields the conclusion.
\end{proof}

Theorem~\ref{thm:affine} implies the following.

\begin{cor}\label{cor:face}
Let $B$ be a maximal lattice-free polytope in $\R^n$.
Then the set $\{\, f\in B: \vol(R(B,f)/\Z^n) = 1 \,\}$ is a face of $B$.
\end{cor}
\begin{proof}
Since
$\vol(R(B,f)/\Z^n)$ is always at most 1, the value 1 is a maximum value for the function
$\vol(R(B,f)/\Z^n)$. By Theorem~\ref{thm:affine}, optimizing this function over $B$ is a linear
 program and hence the optimal set is a face of $B$.
\end{proof}

\begin{theorem}
	\label{inv:thm}
	\thmheader{Unique-lifting invariance theorem}
	Let $n \in \N$ and let $B$ be a maximal lattice-free polytope in $\R^n$. Let $f_1, f_2 \in \intr(B)$. Then $B$ has the unique-lifting property with respect to $f_1$ if and only if $B$ has the unique-lifting property with respect to $f_2$. 
\end{theorem}
\begin{proof}%[Proof of Theorem~\ref{inv:thm}] 

Corollary~\ref{cor:face} implies Theorem~\ref{inv:thm}. Indeed, if the set $\{\, f\in B: \vol(R(B,f)/\Z^n) = 1 \,\}$ is~$B$, then $R(B,f)+\Z^n=\R^n$ for all $f \in B$ and, in particular, for all $f\in \intr(B)$. Hence, $B$ has the unique-lifting property with respect to every $f \in \intr(B)$. Otherwise, $R(B,f)+\Z^n\not=\R^n$ for all $f \in \intr(B)$, which implies that $B$ has the multiple-lifting property with respect to every $f \in \intr(B)$.\end{proof}

\section{Limits of polytopes with the unique-lifting property}\label{limits}

\paragraph{Background information on the Hausdorff metric and convex bodies.} 
We first collect standard notions and basic facts from convex geometry that we need for our topological result on unique liftings. We refer the reader to the monograph \cite[Chapter~1]{MR1216521}.

Let $\cC^n$ be the family of all nonempty compact subsets of $\R^n$ and $\cK^n$ be the family of all nonempty compact \emph{convex} subsets of $\R^n$. With each $K \in \cC^n$ we associate the support function $h(K,\dotvar)$ defined by $h(K,u):= \max \setcond{\sprod{u}{x}}{x \in K}.$
The function $h(K,\dotvar)$ is sublinear and, by this, also continuous. Furthermore, $h(K,u)$ is additive in $K$ with respect to the Minkowski addition. That is, if $K, L \in \cC^n$ and $u \in \R^n$, then $h(K+L,u) = h(K,u) + h(L,u)$.

As a direct consequence of separation theorems the following characterization of the inclusion-relation for elements in $\cK^n$ can be derived: for $K, L \subseteq \cK^n$, one has $K \subseteq L$ if and only if the inequality $h(K,u) \le h(L,u)$ holds for every $u \in \R^n$.  The latter, in combination with the additivity of $h(K,u)$ in $K$, implies the cancellation law for Minkowski addition: if $K,L, M \in \cK^n$, then the inclusion $K \subseteq L$ is equivalent to the inclusion $K+M \subseteq L +M$.

For $K \in \cK^n$, a set of the form 
\begin{equation}
	\label{F in dir u eq}
	F(K,u):= \setcond{x \in K}{\sprod{u}{x} = h(K,u)},
\end{equation}
where $u \in \R^n$, is called an exposed face of $K$ in direction $u$. All (nonempty) faces of a polytope are exposed.

Given $c \in \R^n$ and $\rho \ge 0$, let $\bB(c,\rho)$ denote the closed (Euclidean) ball of radius $\rho$ with center at $c$. For $K,L \in \cC^n$, the Hausdorff distance $\dist(K,L)$ between $K$ and $L$ is defined by 
\[
	\dist(K,L) := \min \setcond{\rho \ge 0}{K \subseteq L + \bB(o,\rho), \ L \subseteq K + \bB(o,\rho)}.
\]
The Hausdorff distance is a metric on $\cC^n$. The formulation of the main theorem of this section involves the topology induced by the Hausdorff distance. In what follows, speaking about the convergence for sequences of elements from  $\cC^n$, we shall always mean the convergence in the Hausdorff metric.

It is known that the mapping $h(K,u)$ continuously depends on the pair $(K,u) \in \cK^n \times \R^n$; see the comment preceding Lemma~1.8.10 in \cite{MR1216521}. More precisely, if $(K_t)_{t \in \N}$ is a sequence of elements of $\cK^n$ converging to some $K \in \cK^n$ and $(u_t)_{t \in \N}$ is a sequence of vectors in $\R^n$ converging to some vector $u \in \R^n$, then $h(K_t,u_t)$ converges to $h(K,u)$, as $t \rightarrow \infty$. Furthermore, one can pass to the limit in inclusions, with respect to convergence in the Hausdorff distance. More precisely, let $(K_t)_{t \in \N}$ and $(L_t)_{t \in \N}$ be sequences of elements from $\cC^n$ converging to $K \in \cC^n$ and $L \in \cC^n$. If $K_t \subseteq L_t$ for every $t \in \N$, then $K \subseteq L$. In particular, considering the case that $K_t$ consists of a single point, say $p_t$, we derive that if $p_t \in L_t$ for every $t \in \N$ and, $p_t$ converges to some $p \in \N$, as $t \rightarrow \infty$, then we have $p \in L$. 

We also note that, if $K, L \in \cC^n$, the definition of the Hausdorff distance implies
\begin{equation}
	\label{dist and conv hull}
	\dist \bigl(\conv(K), \conv(L) \bigr) \le \dist \bigl (K,L \bigr).
\end{equation} 

\paragraph{Topology of the space of polytopes with the unique-lifting property.}  The following is the main result of this section.

\begin{theorem} \label{thm:limit}
	Let $(B_t)_{t \in \N}$ be a convergent sequence (in the Hausdorff metric) of maximal lattice-free polytopes in $\R^n$ such that the limit $B$ of this sequence is a maximal lattice-free polytope. If, for every $t \in \N$, the set $B_t$ has the unique-lifting property, then $B$ too has the unique-lifting property. 
\end{theorem}
\begin{proof} 
	Fix an arbitrary $f \in \intr(B)$. Let $R:=R(B,f)$. We need to verify $R+\Z^n = \R^n$.

	Choose $\eps>0$ such that $\bB(f,\eps) \subseteq B$. Let us show that $f \in B_t$ for all sufficiently large $t \in \N$ and, thus, the lifting region $R_t:=R(B_t,f)$ is well-defined for all sufficiently large $t$. Since $B$ is the limit of $B_t$, as $t \rightarrow \infty$, there exists $t_0 \in \N$ such that $B \subseteq B_t + \bB(o,\eps)$ for all $t \ge t_0$. Hence, $f + \bB(o,\eps)=\bB(f,\eps) \subseteq B \subseteq B_t + \bB(o,\eps)$ for all $t \ge t_0$. Using the cancellation law for the Minkowski addition, we arrive at $f \in B_t$ for all $t \ge t_0$. Thus, replacing $(B_t)_{t \in \N}$ by its appropriate subsequence, we assume that $f \in B_t$ for every $t \in \N$.
	
	Assume that, for every $t \in \N$, the set $B_t$ has the unique-lifting property. By Theorem~\ref{thm:affine}, one has $R_t + \Z^n = \R^n$. Below, we use the latter relation to show  $R + \Z^n = \R^n$. We consider an arbitrary $x \in \R^n$ and show that $x \in R+ \Z^n$.

	Since $x \in \R^n = R_t + \Z^n$, there exists $w_t \in \Z^n$ such that $x \in R_t + w_t$ for every $t \in \N$. Every convergent sequence of nonempty compact sets is necessarily bounded. Hence, there exists $M \in \N$ such that every $B_t$ with $t \in \N$ is a subset of the box $[-M,M]^n$. In view of the inclusions $R_t \subseteq B_t \subseteq [-M,M]^n$, the vector $w_t$ lies in the finite set $([-M,M]^n +x) \cap \Z^n$. Consequently, replacing $(B_t)_{t \in \N}$ by its appropriate subsequence, we can assume that $w_t$ is independent of $t$. With this assumption, we have $x \in R_t + w$ for every $t \in \N$, where $w:=w_t$. The point $x-w$ belongs to $R_t$. Using the definition of the lifting region, we conclude that $x-w \in \conv (\{f\} \cup F_t) \cup (f + z_t - \conv (\{f\} \cup F_t)$ for some face $F_t$ of $B_t$ and some point $z_t \in F_t \cap \Z^n$. Note that $z_t$ lies in the finite set $[-M,M]^n \cap \Z^n$. Passing to appropriate subsequences once again, we can assume that $z_t$ is independent of $t$. With this assumption, for every $t \in \N$, we have
	\begin{equation} \label{x-w formula}
		x-w \in \conv(\{f \} \cup F_t) \cap (f + z - \conv(\{f \} \cup F_t)),
	\end{equation}
	where $z:=z_t$. Relation \eqref{x-w formula} implies that both $x-w$ and $(f+z) - (x - w)$ belong to $\conv(\{f\} \cup F_t$). The latter can be represented by the equalities
	\begin{align}
		x-w & = (1-\lambda_t) f + \lambda_t p_t, \label{x-w for t}
		\\ (f + z) - (x - w) & = (1-\mu_t) f + \mu_t q_t \label{f+z-x+w for t},
	\end{align}
	which hold for some $\lambda_t,\mu_t \in [0,1]$ and some $p_t,q_t \in F_t$.
	We can represent $F_t$ as $F_t=F(B_t,u_t)$ for some unit vector $u_t$ in $\R^n$. The conditions $p_t \in F_t$ and $q_t \in F_t$ can be reformulated as follows:
	\begin{align}
		p_t & \in B_t, & h(B_t,u_t) & = \sprod{u_t}{p_t}, \label{p_t}
		\\ q_t & \in B_t, & h(B_t,u_t) & = \sprod{u_t}{q_t}. \label{q_t}
	\end{align}
	Since $\lambda_t, \mu_t$ lie in the compact set $[0,1]$, the points $p_t, q_t$ lie in the compact set $[-M, M]^n$ and $u_t$ lie in the unit sphere, we can pass to a subsequence and assume that that the scalars $\lambda_t,\mu_t \in [0,1]$, the points $p_t, q_t$ and the unit vector $u_t$ converge to some scalars $\lambda,\mu \in [0,1]$, points $p \in \R^n$, $q \in \R^n$ and a unit vector $u$, respectively, as $t \rightarrow \infty$. Passing to the limit, as $t \rightarrow \infty$, in relations \eqref{x-w for t}--\eqref{q_t}, we arrive at the respective relations
	\begin{align*}
		x-w & = (1-\lambda) f + \lambda p, 
		\\ (f + z) - (x - w) & = (1-\mu) f + \mu q
	\end{align*}
	and 
	\begin{align*}
		p & \in B, & h(B,u) & = \sprod{u}{p}, 
		\\ q & \in B, & h(B,u) & = \sprod{u}{q}. 
	\end{align*}
	The latter implies that 
	$x - w \in \conv( \{f \} \cup F) \cap (f+z - \conv(\{f\} \cup F))$ for $F:=F(B,u)$. Thus, $x - w$ belongs to $R$ and, by this, $x$ belongs to $R+ \Z^n$. Since $x$ was chosen arbitrarily, we arrive at $\R^n = R + \Z^n$. It follows that $B$ has the unique-lifting property.
\end{proof}

\newcommand{\cM}{\mathcal{M}}
\newcommand{\cU}{\mathcal{U}}

Let $\cU^n$ be the set of all maximal lattice-free polytopes in $\R^n$ with the unique-lifting property and let $\cM^n$ be the set of all maximal lattice-free polytopes in $\R^n$. We view $\cU^n$ and $\cM^n$ as metric spaces endowed with the Hausdorff metric. Clearly, $\cU^n \subseteq \cM^n$. Theorem~\ref{thm:limit} asserts that $\cU^n$ is a \emph{closed} subset of $\cM^n$. 

%However, $\cU^n$ is a {\em not} a closed subset of the larger metric space $\cK^n$ of all nonempty compact convex subsets of $\R^n$. Even for $n=2$, one can construct a sequence of maximal lattice-free triangles that have the unique-lifting property, known as Type 2 triangles, that converge to a non-maximal lattice-free set (for example a sequence of Type 2 triangles exist that converge to the set $\{(x_1, x_2): x_1 \geq 0, 0 \leq x_2 \leq 1\}$), so the limit is not even in $\cM^2$. This example also shows that $\cM^2$ is not a closed subset of $\cK^2$. For every $n \ge 3$, one can create even more perverse examples: there exists a sequence of maximal lattice-free sets which get arbitrarily flat and in the limit become non-full-dimensional, following a construction from \cite[Example~3.4]{MR2969261} (we recall that according to our definition, maximal lattice-free sets in $\R^n$ are required to be $n$-dimensional). Thus, $\cM^n$ is not a closed subset of $\cK^n$.
%

In view of Theorem~\ref{thm:limit}, one may wonder if the limit of every convergent sequence of maximal lattice-free polytopes with the unique-lifting property is necessarily a maximal lattice-free set. This question is beyond the scope of this manuscript; however, we note that the limit of a convergent sequence of \emph{arbitrary} maximal lattice-free polytopes is not necessarily a maximal lattice-free polytope. In other words, $\cM^n$ is not a closed subset of $\cK^n$. For every $n \ge 3$, this can be shown by considering a sequence of maximal lattice-free sets which get arbitrarily flat and in the limit become non-full-dimensional, following a construction from \cite[Example~3.4]{MR2969261} (we recall that according to our definition, maximal lattice-free sets in $\R^n$ are required to be $n$-dimensional). In the case $n=2$ one can show that the square $[0,1]^2$, which is lattice-free but not maximal lattice-free, lies in the closure of $\cM^2$. Let $c$ be the center of $[0,1]^2$. Rotating the square by a small angle around $c$ and slightly enlarging the rotated square by a homothetic transformation with homothetic center at $c$, we can construct a maximal lattice-free polytope from $\cM^2$ which is arbitrarily close to $[0,1]^2$. This shows that $[0,1]^2$ belongs to the closure of $\cM^2$. Hence, $\cM^2$ is not a closed subset of $\cK^2$. 

\section{Construction of polytopes with the unique-lifting property}

\label{constr sect}

\paragraph{Coproduct and its properties.} Recall that $\cF(P)$ denotes the set of all faces of $P$, $\cF^i(P)$ denotes the set of faces of dimension~$i$ and $\cF^{\dim(P)-1}(P)$ denotes the set of facets of $P$. Further, $h(P,u):= \max \setcond{\sprod{u}{x}}{x \in P}$ denotes the support function of the polytope $P$, and we will use the notation $F(P,u) =  \setcond{x \in P}{h(P,u) = \sprod{u}{x}}$ to be the optimal face of $P$ when maximizing in the direction of $u$.

Let $n_1, n_2 \in \N$ and $n:=n_1+n_2$. For each $i \in \{1,2\}$, let $o_i$ be the origin of $\R^{n_i}$. %and let $\Lambda_i$ be a translate of a lattice  of rank $n_i$ in $\R^{n_i}$. 

Given $K_1 \in \cK^{n_1}$ and $K_2 \in \cK^{n_2}$, the set 
\[
	K_1 \copr K_2 := \conv ( K_1 \times \{o_2\} \cup \{o_1\} \times K_2) \in \cK^{n}.
\]
is called the \emph{coproduct} of $K_1$ and $K_2$~\footnote{\cite[p.~250]{MR1730169} calls this construction the {\em free sum}; we use {\em coproduct} following a suggestion by Peter McMullen. The construction is dual to the operation of taking Cartesian products, i.e., when $o_i \in \intr(K_i)$ for each $i \in \{1,2\}$, we have the relation $(K_1 \times K_2)^\circ = K_1^\circ \copr K_2^\circ$ for the polar polytopes of $K_1 \times K_2$, $K_1$ and $K_2$.}.  Clearly, up to nonsingular affine transformations, pyramids and double pyramids can be given as coproducts $K_1 \copr K_2$ for polytopes $K_1, K_2$ with $o_i \in K_i$ for each $i \in \{1,2\}$ and $\dim(K_1) = 1$.  Note also that the coproduct operation is associative, and so in the expressions involving the coproduct of three and more sets we can omit brackets. We shall use coproducts of an arbitrary number of sets later in this section.  
  
Clearly,
\begin{equation}
	\label{coproduct as union}
	K_1 \copr K_2  = \bigcup_{0 \le \lambda \le 1} (1-\lambda) K_1 \times \lambda K_2.
\end{equation}
By the basic properties of the relative-interior operation (see \cite[Theorem~6.9]{MR0274683}), we have 
\begin{equation}
	\label{intr coproduct as union}
	\relintr(K_1 \copr K_2) = \bigcup_{0 < \lambda < 1} (1-\lambda) \relintr(K_1) \times \lambda \relintr(K_2).
\end{equation}
If $\dim(K_i) = n_i$ for each $i \in \{1,2\}$, then $\dim(K_1 \copr K_2) = n$ and the operation $\relintr$ in \eqref{intr coproduct as union} can be replaced by $\intr$.

\begin{lemma} 
	\thmheader{On faces of the coproduct of polytopes}
	\label{faces coproduct lem}
	For $i \in \{1,2\}$, let $n_i \in \N$, let $o_i$ be the origin of $\R^{n_i}$ and let $P_i$ be an $n_i$-dimensional polytope in $\R^{n_i}$ with $o_i \in P_i$. Let $P:=P_1 \copr P_2 \subseteq \R^n$, where $n:=n_1+n_2$. Let $F$ be a nonempty subset of $\R^n$. Then the following assertions hold:
	\begin{enumerate}[(a)]
		\item The set $F$ is a face of $P$ if and only if one of the following four conditions is fulfilled:
		\begin{align}
			\label{two faces without}
			F & = F_1 \copr F_2, & &\text{where} &  & o_i \not\in F_i \in \cF(P_i) \ \forall i \in \{1,2\}, & &\text{or}
			\\ 
			\label{two faces with}
			F & = F_1 \copr F_2, & &\text{where} & & o_i \in F_i \in \cF(P_i) \ \forall i \in \{1,2\}, & &\text{or}
			\\ 
			\label{the first face}
			F & = F_1 \times \{o_2\}, & &\text{where} & & o_1 \not\in F_1 \in \cF(P_1),  & &\text{or}
			\\ 
			\label{the second face}
			F & = \{o_1\} \times F_2, & & \text{where} & & o_2 \not\in F_2 \in \cF(P_2).
		\end{align}
		\item Under conditions \eqref{two faces without}, \eqref{two faces with}, \eqref{the first face} and \eqref{the second face}, the dimension of the face $F$ of $P$ is expressed by the equalities $\dim(F) = \dim(F_1) + \dim(F_2) +1$, $\dim(F) = \dim(F_1) + \dim(F_2)$, $\dim(F) = \dim(F_1)$ and $\dim(F) = \dim(F_2)$, respectively.
		\item The set $F$ is a facet of $P$ if and only if one of the following three conditions is fulfilled:
		\begin{align}
			\label{two facets without}
			F & = F_1 \copr F_2, & & \text{where} & & o_i \not\in F_i \in \cF^{n_i-1} (P_i) \ \forall i \in \{1,2\}, & &\text{or}
			\\ 
			\label{the first facet}
			F & = F_1 \copr P_2, & & \text{where} & & o_1 \in F_1 \in \cF^{n_1-1} (P_1), & &\text{or}
			\\ 
			\label{the second facet}
			F & = P_1 \copr F_2, & & \text{where} & & o_2 \in F_2 \in \cF^{n_2-1}(P_2).
		\end{align}
	\end{enumerate}
\end{lemma}
\begin{proof}	
	\emph{Assertion~(a):} We start with the necessity in (a). We assume that $F$ is a face of $P$ and show that one of the four conditions \eqref{two faces without}--\eqref{the second face} is fulfilled. By the assumptions of the lemma, $F$ is nonempty, and so we have 
	\[
		F=F(P,u) =  \setcond{x \in P}{h(P,u) = \sprod{u}{x}}
	\]
	for some $u :=(u_1,u_2) \in \R^{n_1} \times \R^{n_2}$. Since
	\[
		h(P,u) = \max \{ h(P_1,u_1), h(P_2,u_2)\},
	\]
	we get: 
	\begin{align}
		\label{face of coproduct eq =}
		F(P,u) & = F(P_1,u_1) \copr F(P_2,u_2) 
			& & \text{if} & h(P_1,u_1) &=h(P_2,u_2),
		\\
		\label{face of coproduct eq >}
		F(P,u) & =	F(P_1,u_1) \times \{o_2\}
			& & \text{if} & h(P_1,u_1) & > h(P_2,u_2), 
		\\
		\label{face of coproduct eq <}
		F(P,u) &= \{o_1\} \times F(P_2,u_2)
			& & \text{if} & h(P_1,u_1) & < h(P_2,u_2).
	\end{align}
	
	Since for each $i \in \{1,2\}$, one has $o_i \in P_i$, the support function $h(P_i,\dotvar)$ is nonnegative. We compare $h(P_1,u_1)$ with $h(P_2,u_2)$ and distinguish three cases as in \eqref{face of coproduct eq =}--\eqref{face of coproduct eq <}.
	
	In the case $h(P_1,u_1)=h(P_2,u_2)$, we define $F_i := F(P_i,u_i)$ for $i \in \{1,2\}$. One has either $h(P_1,u_1)=h(P_2,u_2)>0$ or $h(P_1,u_1)=h(P_2,u_2)=0$. In the former case, the face $F(P_i,u)$ of $P_i$ does not contain $o_i$ for each $i \in \{1,2\}$ and, by this, \eqref{two faces without} is fulfilled. In the latter case, the face $F(P_i,u_i)$ of $P_i$ contains $o_i$ for each $i \in \{1,2\}$ and, by this, \eqref{two faces with} is fulfilled. If $h(P_1,u_1) > h(P_2,u)$, we have $h(P_1,u_1) > 0$ and so \eqref{the first face} is fulfilled for $F_1 := F(P_1,u_1)$. Analogously, if $h(P_2,u_1) < h(P_2,u)$, we have $h(P_2,u_2)>0$ and so \eqref{the second face} is fulfilled for $F_2 := F(P_2,u_2)$. This proves the necessity in (a).
		
	For proving the sufficiency, we assume that $F$ fulfills one of the four conditions \eqref{two faces without}--\eqref{the second face} and show that $F=F(P,u)$ for an appropriate choice of $u \in \R^n$. Consider the case that $F$ fulfills \eqref{two faces without}. For each $i \in \{1,2\}$, we choose a vector $u_i \in \R^{n_i}$, with $F_i = F(P_i,u_i)$. Since $o_i \not \in F_i$, we have $h(P_i,u_i) > 0$ and by this also $u_i \ne o_i$. Appropriately rescaling the vectors $u_1$ and $u_2$, we ensure the inequality $h(P_1,u_1) = h(P_2,u_2)$. In view of \eqref{face of coproduct eq =}, it follows that $F=F(P,u)$ with $u=(u_1,u_2)$. The case that $F$ fulfills \eqref{two faces with} is similar. For each $i \in \{1,2\}$ we choose $u_i \in \R^{n_i}$ with $F_i = F(P_i,u_i)$. Since $o_i \in F_i$, we have $h(P_i,u_i)=0$.  In view of \eqref{face of coproduct eq =}, $F=F(P,u)$ with $u=(u_1,u_2)$. In the case that $F$ fulfills \eqref{the first face}, we choose $u_1 \in \R^{n_1}$ satisfying $F_1=F(P_1,u_1)$ and define $u_2:=o_2$. Since $o_1 \not\in F_1$, we get $h(P_1,u_1)>0$. In view of \eqref{face of coproduct eq >}, we get $F=F(P,u)$ for $u=(u_1,u_2)$. The case that $F$ fulfills \eqref{the second face} is completely analogous to the previously considered case.

	\emph{Assertion~(b):} The cases \eqref{the first face} and \eqref{the second face} are trivial. In the case \eqref{two faces with}, the face $F$ contains $o$, and so the dimension of $F$ is the dimension of the linear hull of $F$. Applying the definition of the coproduct and the fact that, for each $i \in \{1,2\}$, the face $F_i$ contains $o_i$, we see that the linear hull of $F$ is the Cartesian product of the linear hulls of $F_1$ and $F_2$. Hence $\dim(F) = \dim(F_1) + \dim(F_2)$. The case \eqref{two faces without} can be handled by a reduction to the case \eqref{two faces with}. Assume that \eqref{two faces without} is fulfilled. Then, in view of \eqref{coproduct as union}, the set $F= F_1 \copr F_2$ does not contain $o$. Then $\dim(K) = \dim(F)+1$ for $K := \conv(F \cup \{o\})$. Using the definition of the coproduct, we can easily verify the equality $K = \conv (F_1 \cup \{o_1\}) \copr \conv(F_2 \cup \{o_2\})$. Using the argument from the case \eqref{two faces with} we conclude that $\dim (K) = \dim(\conv(F_1 \cup \{o_1\}) + \dim (\conv(F_2 \cup \{o_2\})$. Since $o_i \not\in F_i$, we have $\dim(\conv(F_i \cup \{o_i\})) = \dim(F_i) + 1$. Hence $\dim(F) = \dim(F_1) + \dim(F_2) + 1$.
	
	\emph{Assertion~(c)} is a straightforward consequence of (a) and (b).
\end{proof}

\paragraph{Constructions based on the coproduct operation.}  
It will be more convenient to work with general affine lattices and then specialize to the $\Z^n$ case. To that end, we say that a set $\Lambda \subseteq \R^n$ is an affine lattice of rank $n$ if $\Lambda$ is a translation of a lattice of rank $n$. Equivalently, a set $\Lambda \subset \R^n$ is an affine lattice of rank $n$ if and only if there exist affinely independent points $x_0,\ldots,x_n \in \R^n$ such that $\Lambda$ is the set of all $x=z_0 x_0 + \cdots + z_n x_n$ with $z_0,\ldots,z_n \in \Z$ and $z_0 + \cdots + z_n =1$. %Clearly, the notions that we introduced above for an arbitrary lattice of rank $\Lambda$ can be carried over directly to the case of an arbitrary affine lattice $\Lambda$. 
Note that the notions of lattice-free set, maximal lattice-free set, lifting region (denoted by $R(B,f)$) and sets with the unique-lifting property, which were introduced with respect the integer lattice $\Z^n$, can be extended directly to the more general situation where, in place of $\Z^n$, we take an arbitrary affine lattice $\Lambda$ of rank $n$ in $\R^n$. Thus, for such $\Lambda$ we can introduce the respective notions of \emph{$\Lambda$-free set}, \emph{maximal $\Lambda$-free set}, \emph{lifting region with respect to $\Lambda$} (which we will denote by $R_\Lambda(B,f)$) and \emph{set with the unique-lifting property with respect to $\Lambda$}. %Even more generally it will be more convenient to work in the `affine setting' with respect to $\Lambda$. That is, we say that a set $\Lambda \subseteq \R^n$ is an affine lattice of rank $n$ if $\Lambda$ is a translation of a lattice of rank $n$. Affine lattices can also be described intrinstically as follows. A set $\Lambda \subset \R^n$ is an affine lattice of rank $n$ if there exist affinely independent points $x_0,\ldots,x_n \in \R^n$ such that $\Lambda$ is the set of all $x=z_0 x_0 + \cdots + z_n x_n$ with $z_0,\ldots,z_n \in \Z$ and $z_0 + \cdots + z_n =1$. Clearly, the notions that we introduced above for an arbitrary lattice of rank $\Lambda$ can be carried over directly to the case of an arbitrary affine lattice $\Lambda$. 
We shall use the following result of Lov\'asz \cite{MR1114315}; see  also \cite{MR2724071} and \cite{MR3027668}.

\begin{theorem} \thmheader{Lov\'asz's characterization of maximal lattice-free sets}
	\label{lovasz thm}
	Let $B$ be an $n$-dimensional $\Lambda$-free polyhedron in $\R^n$. Then $B$ is maximal $\Lambda$-free if and only if $\relintr(F) \cap \Lambda \ne \emptyset$ for every facet $F$ of $B$.
\end{theorem}

The following is the main result of this section.

\begin{theorem}
	\thmheader{Coproduct construction of various types of $\Lambda$-free sets}
	\label{coproduct thm}
	For $i \in \{1,2\}$, let $n_i \in \N$, let $o_i$ be the origin of $\R^{n_i}$, let $\Lambda_i$ be an affine lattice of rank $n_i$ in $\R^{n_i}$ and let $B_i$ be an $n_i$-dimensional polytope with $o_i \in B_i$. Let $0 < \mu < 1$. Then, for the $n$-dimensional polytope $B:=B_1 \copr B_2$ with $n : = n_1 + n_2$ and the affine lattice $\Lambda := (1-\mu) \Lambda_1 \times \mu \Lambda_2$ of rank $n$, the following assertions hold:
	\begin{enumerate}[(a)]
		\item If $B_i$ is $\Lambda_i$-free for each $i \in \{1,2\}$, then $B$ is $\Lambda$-free.
		\item If $B_i$ is maximal $\Lambda_i$-free for each $i \in \{1,2\}$, then $B$ is maximal $\Lambda$-free.
		\item If $B_i$ is maximal $\Lambda_i$-free and has the unique-lifting property with respect to $\Lambda_i$ for  each $i \in \{1,2\}$, then $B$ is maximal $\Lambda$-free and has the unique-lifting property with respect to $\Lambda$.
	\end{enumerate}
\end{theorem}
\begin{proof}
	\emph{(a):} Consider an arbitrary point $x=(x_1,x_2) \in \R^{n_1} \times \R^{n_2}$ belonging to $\intr(B_1 \copr B_2)$. In view of \eqref{intr coproduct as union}, one has $x_1 \in (1-\lambda) \intr(B_1)$ and $x_2 \in \lambda \intr(B_2)$ for some $0 < \lambda < 1$. If $\lambda \ge \mu$, then taking into account $o_1 \in B_1$, we obtain $x_1 \in (1-\lambda) \intr(B_1) \subseteq (1-\mu) \intr(B_1)$. Since $B_1$ is $\Lambda_1$-free, we have $x_1 \not\in (1-\mu)\Lambda_1$ implying that $x \not\in \Lambda$. Analogously, in the case $\lambda \le \mu$, we deduce that the point $x_2$ is not in $\mu\Lambda_2$ and thus $x \not\in \Lambda$. %It follows that $x \not\in \Lambda$, independently on the value of $\lambda$.
		
	\emph{(b):} Assume that $B_i$ is maximal $\Lambda_i$-free for each $i \in \{1,2\}$. We show that $B$ is maximal $\Lambda$-free. In view of (a), the polytope $B$ is $\Lambda$-free.  By Theorem~\ref{lovasz thm}, in order to verify the maximality of $B$ it suffices to show that the relative interior of each facet of $B$ contains a point of $\Lambda$.  Let $F$ be an arbitrary facet of $B$. We use the classification of facets of the coproduct which is provided by Lemma~\ref{faces coproduct lem}.(c). Consider the case that $F$ fulfills condition \eqref{two facets without}. By Theorem~\ref{lovasz thm} applied to the maximal $\Lambda_i$-free set $B_i$, for each $i \in \{1,2\}$ there exists a point $x_i \in \relintr(F_i)$ belonging to $\Lambda_i$. Then the point $x := ((1-\mu) x_1, \mu x_2)$ belongs to $\Lambda$. By \eqref{intr coproduct as union}, the point $x$ also belongs to $\relintr(F)$. Let us switch to the case that $F$ fulfills condition \eqref{the first facet}. By Theorem~\ref{lovasz thm}, there exists a point $x_1 \in \relintr(F_1)$ which belongs to $\Lambda_1$. Then $(1-\mu) x_1 \in (1-\mu) \relintr(F_1)$. Since $o_1 \in F_1$, the latter containment relation remains valid if we slightly shrink the right hand side $(1-\mu) \relintr(F_1)$. That is, $(1-\mu) x_1 \in (1-\lambda) \relintr(F_1)$ for
	some $\lambda$ satisfying $\mu < \lambda < 1$, which is sufficiently close to $\mu$. Since $o_2 \in B_2$, we have $\intr(B_2) \varsubsetneq \frac{\lambda}{\mu} \intr(B_2)$. Since $B_2$ is maximal $\Lambda_2$-free, there exists a point $x_2 \in \frac{\lambda}{\mu} \intr(B_2)$ which belongs to $\Lambda_2$. It follows that $\mu x_2 \in \lambda \intr(B_2)$ is a point belonging to $\mu \Lambda_2$. Thus, $x:=((1-\mu) x_1, \mu x_2)$ is a point belonging to $(1-\lambda) \relintr(F_1) \times \lambda \intr(B_2)$ and to $\Lambda$. Taking into account \eqref{intr coproduct as union}, we see that $x$ belongs to $\relintr(F)$. The case of $F$ fulfilling condition \eqref{the first facet} is completely analogous to the previously considered case. Summarizing, we conclude that the relative interior of each facet of $B$ contains a point of $\Lambda$. Thus, by Theorem~\ref{lovasz thm}, the set $B$ is maximal $\Lambda$-free.
	
	\emph{(c):} We distinguish two cases. 

	\emph{Case~1: $o_i \in \intr(B_i)$ for each $i \in \{1,2\}$.} In this case $o \in \intr(B)$. Thus for showing that $B$ has the unique-lifting property with respect to $\Lambda$, it suffices to check the equality $R  + \Lambda = \R^n$ for the lifting region
	\begin{equation}
		\label{R expression}
		R := R_\Lambda(B,o) = \bigcup_{\stackrel{F \in \cF^{n- 1}(B)}{z \in \Lambda \cap F} } S_{F,z}(o),
	\end{equation}
	where, for every $F \in \cF^{n - 1}(B)$ and $z \in \Lambda \cap F$, one has
	\begin{equation}
		\label{spindle wrt o}
		S_{F,z}(o) = \conv (F \cup \{o\}) \cap \bigl(z - \conv(F \cup \{o\}) \bigr).
	\end{equation}  
	Since $o_i \in \intr(B_i)$, no facet of $B_i$ contains $o_i$. Thus, by Lemma~\ref{faces coproduct lem}.(c), a subset $F$ of $\R^n$ is a facet of $B$ if and only if $F = F_1 \copr F_2$, where $F_i$ is a facet of $B_i$ for each $i \in \{1,2\}$. We consider an arbitrary such facet $F=F_1 \copr F_2$. Choose also an arbitrary $z_i \in F_i \cap \Lambda_i$ for each $i \in \{1,2\}$ and introduce the point $z:=((1-\mu) z_1, \mu z_2)$, which by construction belongs to $F \cap \Lambda$. We establish an inclusion relation between $S_{F,z}(o)$ and the two sets $S_{F,z_i}(o_i)$ with $i \in \{1,2\}$. The set $\conv(F \cup \{o\})$, which occurs twice on the right hand side of \eqref{spindle wrt o}, fulfills
	\begin{align}
		\nonumber
		\conv(F \cup \{o\}) 
		& = \conv \bigl(F_1 \times \{o_2\} \cup \{o_1\} \times F_2 \cup \{o\} \bigr) 
		\\ \nonumber & = \conv \bigl( (F_1 \cup \{o_1\}) \times \{o_2\} \cup \{o_1\} \times (F_2 \cup \{o_2\}) \bigr) 
		\\ \nonumber & = \conv(F_1 \cup \{o_1\}) \copr \conv(F_2 \cup \{o_2\})
		\\ & \supseteq (1-\mu) \conv(F_1 \cup \{o_1\}) \times \mu \conv(F_2 \cup \{o_2\}). \label{incl}
	\end{align}
	
Therefore, also $z - \conv(F \cup \{o\}) \supseteq (1-\mu)(z_1 - \conv(F_1 \cup \{o_1\})\times \mu(z_2 - \conv(F_2 \cup \{o_2\}))$. Analogously to \eqref{spindle wrt o}, one has $S_{F_i,z_i}(o_i) = \conv(F_i \cup \{o_i\}) \cap \bigl (z_i - \conv(F_i \cup \{o_i\}\bigr)$ for each $i \in \{1,2\}$. 
	From this and \eqref{incl}, we get
	\begin{align*}
		S_{F,z}(o) \supseteq (1-\mu) S_{F_1,z_1}(o_1) \times \mu S_{F_2,z_2}(o_2).
	\end{align*}
	In view of \eqref{R expression}, the latter yields the following relation between $R$ and the lifting regions $R_i:= R_{\Lambda_i} (B_i,o_i)$ with $i \in \{1,2\}$:
	\[
		R \supseteq (1-\mu) \bigcup_{\stackrel{F_1 \in \cF^{n_1-1}(B_1)}{z_1 \in \Lambda_1 \cap F_1}} S_{F_1,z_1}(o_1) \times \mu \bigcup_{\stackrel{F_2 \in \cF^{n_2-1}(B_1)}{z_2 \in \Lambda_2 \cap F_2}} S_{F_2,z_2}(o_2) = (1-\mu) R_1 \times \mu R_2.
	\]
	Consequently, using the fact that $B_i$ has the unique-lifting property with respect to $\Lambda_i$ and, by this, $R_i + \Lambda_i = \R^{n_i}$ for each $i \in \{1,2\}$, we obtain
	\begin{align*}
		R + \Lambda 
		& \supseteq (1-\mu) R_1 \times \mu R_2 + (1-\mu) \Lambda_1 \times \mu \Lambda_2 
		\\ & = (1-\mu) (R_1 + \Lambda_1) \times \mu (R_2 + \Lambda_2) 
		= (1-\mu) \R^{n_1} \times \mu \R^{n_2} = \R^n.
	\end{align*}
	Thus, $R+\Lambda = \R^n$, and so $B$ has the unique-lifting property with respect to $\Lambda$.
	
	\emph{Case~2: $o_i \not \in \intr(B_i)$ for some $i \in \{1,2\}$.} For each $i \in \{1,2\}$, we have $o_i \in B_i$. Thus we can choose a sequence  $(x_{i,t})_{t \in \N}$ of points in $\intr(B_i)$ converging to $o_i$. For every $i \in \{1,2\}$ and $t \in \N$, the interior of $B_i - x_{i,t}$ contains $o_i$. Clearly, the set $B_i - x_{i,t}$ is maximal $(\Lambda_i - x_{i,t})$-free and has the unique-lifting property with respect to $\Lambda_i - x_{i,t}$. Hence, we can apply the assertion obtained in Case~1. It follows that the set $(B_1 - x_{1,t}) \copr (B_2 - x_{2,t})$ is maximal $(1-\mu) (\Lambda_1 - x_{1,t}) \times \mu (\Lambda_2- x_{2,t})$-free and has the unique-lifting property with respect to this affine lattice. We introduce the vector 
	\[
		x_t:= \bigl((1-\mu) x_{1,t}, \mu x_{2,t} \bigr),
	\]
	and the set $B_t:=(B_1 - x_{1,t}) \copr (B_2 - x_{2,t}) + x_t$ (note that for $t \in \{1,2\}$, there is a collision of notations, because $B_1$ and $B_2$ are already introduced; we avoid this collision by imposing the additional condition $t \ge 3$). The set $B_t$ is maximal $\Lambda$-free and has the unique-lifting property with respect to $\Lambda$. We check that $B_t \rightarrow B = B_1 \copr B_2$, as $t \rightarrow \infty$. 
	Using the notation
	\begin{align*}
		B'_{1,t} & := (B_1 - x_{1,t} ) \times \{o_2\}, & B_1' & := B_1 \times \{o_2\}, \\
		B'_{2,t} & := \{o_1\} \times (B_2 - x_{2,t}), & B_2' & := \{o_1\} \times B_2,
	\end{align*}
	we obtain the following upper bounds on $\dist(B_t,B)$:
	\begin{align*}
		\dist(B_t,B)  = & \dist \bigl(\conv(B_{1,t}' \cup B_{2,t}') + x_t, \conv(B_1' \cup B_2') \bigr) 		
		\\ \le & \dist \bigl(\conv(B_{1,t}' \cup B_{2,t}') , \conv(B_1' \cup B_2') \bigr) + \|x_t\|
		\\ \le & \max \left\{ \dist(B_{1,t}',B_1'), \dist(B_{2,t}',B_2') \right\} + \|x_t\|
	\end{align*}
	One has
	\[
		\dist(B_{i,t}',B_i') = \dist(B_i -x_{i,t}, B_i) \le \|x_{i,t}\|.
	\]
	Thus, $\dist(B_t,B) \le \max \left\{ \|x_{1,t} \| , \|x_{2,t}\|\right\}  + \|x_t\|$, where the right hand side of this equality converges to $0$, as $t \rightarrow \infty$. We have shown that the maximal $\Lambda$-free set $B_t$, which has the unique-lifting property with respect to $\Lambda$, converges to the maximal $\Lambda$-free set $B$, as $t \rightarrow \infty$. By Theorem~\ref{thm:limit}, we conclude that $B$ has the unique-lifting property with respect to $\Lambda$.
\end{proof}

Theorem~\ref{coproduct thm} can be extended to a version dealing with the coproduct of $k \in \N$ sets. 

\begin{cor} 
	\label{coproduct of k sets cor}
Let $k \in \N$. For each $i \in \{1,\ldots,k\}$, let $\mu_i > 0$, let $n_i \in \N$, let $\Lambda_i$ be an affine lattice of rank $n_i$ in $\R^n_i$ and let $B_i$ be a $n_i$-dimensional polytope in $\R^{n_i}$ such that the origin of $\R^{n_i}$ is contained in $B_i$. Then, for  $B = \mu (B_1 \copr \ldots \copr B_k)$ with $\mu := \mu_1 + \cdots + \mu_k$ and $\Lambda := \mu_1 \Lambda_1 \times \cdots \times \mu_k \Lambda_k$, the following assertions hold:
\begin{enumerate}[(a)]
	\item If $B_i$ is $\Lambda_i$-free for each $i \in \{1,\ldots,k\}$, then $B$ is $\Lambda$-free.
	\item If $B_i$ is maximal $\Lambda_i$-free for each $i \in \{1,\ldots,k\}$, then $B$ is maximal $\Lambda$-free.
	\item If $B_i$ is maximal $\Lambda_i$-free and has the unique-lifting property with respect to $\Lambda_i$ for each $i \in\{1,\ldots,k\}$, then $B$ is maximal $\Lambda$-free and has the unique-lifting property with respect to $\Lambda$.
\end{enumerate}
\end{cor}
\begin{proof}
	The assertion follows by induction, by using Theorem~\ref{coproduct thm} in the inductive step, with the basis case $k=1$ being trivial.
\end{proof}

The following is a simple reformulation of Corollary~\ref{coproduct of k sets cor} in a form which uses lattice-free sets rather than general $\Lambda$-free sets, where $\Lambda$ is a translate of an arbitrary lattice.

\begin{cor} \thmheader{Coproduct construction of various types of lattice-free sets}
	\label{coproduct lat free cor}
	Let $k \in \N$. For $i \in \{1,\ldots,k\}$, let $\mu_i>0$, let $n_i \in \N$, let $B_i$ be an $n_i$-dimensional polytope in $\R^{n_i}$ and let $c_i \in B_i$. Then, for the polytope
	\[
		B:= \frac{\mu (B_1-c_1)}{\mu_1} \copr \cdots \copr \frac{\mu (B_k-c_k)}{\mu_k} + (c_1,\ldots,c_k)
	\]
	with $\mu := \mu_1 + \cdots + \mu_k$, the following assertions hold:
	\begin{enumerate}[(a)]
		\item If $B_i$ is lattice-free for each $i \in \{1,\ldots,k\}$, then $B$ is lattice-free.
		\item If $B_i$ is maximal lattice-free for each $i \in \{1,\ldots,k\}$, then $B$ is maximal lattice-free.
		\item If $B_i$ is maximal lattice-free and has the unique-lifting property for each $i \in \{1,\ldots,k\}$, then $B$ is maximal lattice-free and has the unique-lifting property.
	\end{enumerate}
\end{cor}
\begin{proof}
	For every of the three assertions the proof is based on the respective assertion of Corollary~\ref{coproduct of k sets cor}. We only prove (a), since the proofs of (b) and (c) are analogous. Assume that, for each $i \in \{1,\ldots,k\}$, the set $B_i$ is maximal lattice-free. Then
	$B_i-c_i$ is maximal $(\Z^{n_i} - c_i)$-free. The origin of $\R^{n_i}$ belongs to $B_i-c_i$. Thus, we can use Theorem~\ref{coproduct thm} for the sets $B_i-c_i$. We obtain that the set $\mu \bigl( (B_1 - c_1) \copr \cdots \copr (B_k - c_k) \bigr)$ is maximal $\mu_1 (\Z^{n_1} - c_1) \times \cdots \times \mu_k (\Z^{n_k} - c_k)$-free. Using the affine transformation that sends $(x_1,\ldots,x_k) \in \R^{n_1} \times \cdots \times \R^{n_k}$ to $(\frac{1}{\mu_1} x_1, \ldots, \frac{1}{\mu_k} x_k) + (c_1, \ldots, c_k)$, we conclude that the set $B$ is lattice-free.
\end{proof}

\paragraph{Pyramids and double pyramids.} Since pyramids and double pyramids can be described using the coproduct operation,  Corollary~\ref{coproduct lat free cor} can be used to construct pyramids and double pyramids which have the unique-lifting property. This is presented in the following corollary.

\begin{cor}
	Let $B$ be an $n$-dimensional polytope in $\R^n$. Let $c \in B$, let $0 \le \gamma < 1$ and let $0 < \mu < 1$. Then, for the polytope
	\begin{equation}\label{eq:pyr-construct}
		P :=  \conv \left( \frac{B- \mu c}{1-\mu} \times \{\gamma\} \cup \{c\} \times \Bigl[  \frac{(\mu-1)\gamma}{\mu} , \frac{(\mu-1)\gamma+1}{\mu} \Bigr]  \right) 
	\end{equation}
	(which is a pyramid if $\gamma=0$ and a double pyramid otherwise), the following assertions hold:
	\begin{enumerate}[(a)]
		\item If $B$ is lattice-free, then $P$ is lattice-free.
		\item If $B$ is maximal lattice-free, then $P$ is maximal lattice-free. 
		\item If $B$ has the unique-lifting property, then $P$ has the unique-lifting polytope. 
	\end{enumerate}
\end{cor}
\begin{proof}
	A straightforward computation shows that $P = \frac{B-c}{1-\mu} \copr \frac{[0,1]-\gamma}{\mu} + (c,\gamma)$, where $c \in B$ and $\gamma \in [0,1).$ Thus, the assertion follows directly from Corollary~\ref{coproduct lat free cor}.
\end{proof}

We can also use Corollary~\ref{coproduct lat free cor} to provide families of simplices and cross-polytopes having the unique-lifting property. 

\begin{cor}\label{cor:a1-an}
	Let $n \in \N$ and let $a_1,\ldots,a_n >0$ be such that $\frac{1}{a_1} + \cdots + \frac{1}{a_n} = 1$. Then the following assertions hold:
	\begin{enumerate}[(a)]
		\item The simplex $\conv \{o, a_1 e_1,\ldots,a_n e_n \}$ has the unique-lifting property.
		\item The cross-polytope $\conv \{ \pm \frac{a_1}{2} e_1,\ldots, \pm \frac{a_n}{2} e_n \} + (\frac{1}{2},\ldots, \frac{1}{2})$ has the unique-lifting property.
	\end{enumerate}
\end{cor}
\begin{proof}
	For both assertions we use Corollary~\ref{coproduct lat free cor} with $k=n$, $B_i = [0,1]$ and $\mu_i = \frac{1}{a_i}$ for every $i$. For assertion (a) we choose $c_1 = \cdots = c_n = 0$, while for assertion (b) we choose $c_1 = \cdots = c_n = \frac{1}{2}$.
\end{proof}

\begin{remark}
As mentioned in the Introduction, without loss of generality we only deal with polytopes in this paper because maximal lattice-free polyhedra have a recession cone which is a rational linear subspace, and so the lifting region is a cylinder over the lifting region of a polytope. To use the coproduct operation on such unbounded maximal lattice-free polyhedra, one would use the operation on the polytopes - after removing the linearity spaces of the polyhedra - and then add back the direct sum of these two linear subspaces. %Applying the coproduct directly on the polyhedra will not result in polyhedra (they will give sets whose {\em topological closures} will be polyhedral).
\end{remark}

\paragraph{Comparison with existing results on the unique-lifting property.}
All known classes of maximal lattice-free polytopes with the unique-lifting property from the literature can be constructed using the formula~\eqref{eq:pyr-construct}. For simplices with the unique-lifting property:
\begin{itemize}
\item Setting $a_1 = a_2 = \ldots = a_n = n$ gives the so-called Type 1 triangle (for $n=2$) and its higher-dimensional generalizations that were first shown to have the unique-lifting property in~\cite{ccz} and~\cite{bcm}. 
\item All the results on 2-partitionable simplices from Section 4 in~\cite{bcm} can be derived using~\eqref{eq:pyr-construct}.
\end{itemize}
 
Of course, \eqref{eq:pyr-construct} can be used to create pyramids that are not simplices (for example, by creating a $n$-dimensional pyramid over a 2-dimensional quadrilateral with unique-lifting), and so it is a much more powerful and general construction compared to existing constructions for simplices with unique lifting. Similarly, the cross-polytope construction in Corollary~\ref{cor:a1-an} (b) for $n=2$ gives precisely the quadrilaterals with unique-lifting property. Further, {\em every} maximal lattice-free polytope for $n=2$ can be obtained using the coproduct construction~\eqref{eq:pyr-construct}. 

In summary, the coproduct construction can be used to obtain every previously known maximal lattice-free polytope with the unique-lifting property, and gives a very general way to obtain new unique-lifting polytopes in higher dimensions.

\paragraph{The cube construction.} One may wonder, in the light of previous remarks, whether for a given dimension $n \ge 2$ the coproduct construction generates all unique-lifting polytopes in $\R^n$. We saw that for $n =2$ this is indeed the case. Below we describe a construction, which shows that for infinitely many choices of $n$, there exist unique-lifting polytopes which cannot be generated using the coproduct construction. First, we observe that nonsingular affine transformations of cubes of dimension at least three are not coproducts of any polytopes:

\begin{prop}\label{prop:coproduct-nogo} Let $n \in \N$ and $n \ge 3$. Let $B \subseteq [0,1]^n$ an image of the $n$-dimensional cube under a nonsingular affine transformation. Then $B$ is not representable as $P_1 \copr P_2$, where $P_i$ is an $n_i$-dimensional polytope in $\R^{n_i}$ and $n_i \in \N$ for each $i \in \{1,2\}$.
\end{prop}
\begin{proof}
	Assume the contrary, that is, $B = P_1 \copr P_2$. For each $i \in \{1,2\}$ choose a facet $F_i$ of $P_i$ with $o_i \not\in F_i$, where $o_i$ denotes the origin of $\R^{n_i}$. By Lemma~\ref{faces coproduct lem}, the polytopes $F = F_1 \copr F_2$ is a facet of $B$, while the polytopes $F_1 \times \{o_2\}$ and $\{o_1\} \times F_2$ are faces of $B$. All faces of $B$ are nonsingular affine images of cubes of dimensions at most $n$. A cube of dimension $k \in \N$ has $2 k$ facets. Thus, $F_i$ has $2 (n_i-1)$ facets. In the case that $n_i \ge 2$ for each $i \in \{1,2\}$, the set $\cF^{n-2}(F)$ of all facets of $F$ is precisely the set of polytopes of the form $F_1 \copr G_2$ and $G_1 \copr F_2$, where $G_i$ is a facet of $F_i$. Consequently, $F$ has $2(n_1-1) + 2(n_2-1) = 2 (n - 2)$ facets. On the other hand, since $F$ is a facet of $B$ and by this an nonsingular affine image of an $(n-1)$-dimensional cube, $F$ has $2(n-1)$ facets, which is a contradiction. We switch to the case that $n_i=1$ for some $i \in \{1,2\}$. Without loss of generality, let $n_2=1$. In this case $F_2$ is $0$-dimensional and thus $F$ is a pyramid with the base $F_1 \times \{o_2\}$, which is a polytope with $2 (n_1-1)$ facets, and the apex $\{o_1\} \times F_2$. It follows that $F$ has $2 (n_1-1) + 1 = 2(n-2) + 1$ facets, which is a contradiction to the fact that $F$ has $2(n-1)$ facets.
\end{proof}

\begin{prop}\label{prop:cube-const}
Let $n \in \N$ be odd and $n \ge 3$. Let $\Lambda$ be the lattice of rank $n$ in $\R^n$ given by 
\[
	\Lambda := \setcond{ z = (z_1,\ldots,z_n) \in \Z^n}{z_1 + \cdots + z_n \ \text{is even}}.
\]
Then the cube $B:=[0,2]^n$ is a maximal $\Lambda$-free polytope with the unique-lifting property with respect to $\Lambda$. 
\end{prop}
\begin{proof}
Clearly, $B$ is $\Lambda$-free because the only integer point in the relative interior of $B$ is $(1,\ldots,1) \in \R^n$. Since the dimension $n$ is odd, this point is not in $\Lambda$. The cube $B$ is even maximal $\Lambda$-free, because in the relative interior of every facet of $B$ one can find a point of $\Lambda$ with one component equal to $0$ or $2$ and the remaining components equal to $1$. For the verification of the unique-lifting property we choose $f:=(1,\ldots,1)$ and test whether $R_\Lambda(B,f) + \Lambda = \R^n$. Let $D$ be the fundamental parallelepiped for $\Lambda$, i.e., every point in $\R^n$ is uniquely representable as a point from $\Lambda$ and a point from $D$. Define $R_\Lambda(B,f)/ \Lambda = \{p \in D: (p + \Lambda) \cap R_\Lambda(B,f) \neq \emptyset\}$. Then $R_\Lambda(B,f) + \Lambda = \R^n$ if and only if $\vol(R_\Lambda(B,f)/ \Lambda) = \vol(D) = \det(\Lambda) = 2$. Consider the spindle $S:=S_{F,z}(f)$ associated with the facet $F :=[0,2]^{n-1} \times \{0\}$ and the lattice point $z:=(1,\ldots,1,0) \in \Lambda$, where $z$ is the only point of $\Lambda$ in the relative interior of $F$. Furthermore, by symmetry reasons, the parts of $R_\Lambda(B,f)$ associated to the remaining facets of $B$ the same properties with respect to $\Lambda$ as our fixed facet $F$. It follows, taking into account \eqref{vol is sum of facet parts:eq}, that $\vol(R_\Lambda(B,f) / \Lambda) = 2 n \vol(S / \Lambda)$, where $2n$ is the number of facets of $B$. 
We have $\relintr(F) \cap \Lambda = \{z\}$. Thus, in view of \eqref{struct congr:eq}, 
 the interior of $S$ does not contain distinct points congruent modulo $\Lambda$ (note that we apply \eqref{struct congr:eq} with $\Lambda$ in place of $\Z^n$). Thus, one has $\vol(S / \Lambda) = \vol(S)$. Note that, since $F$ is centrally symmetric with center at $z$, we see that $S$ is a double pyramid, which can be represented by $S_{F,z}(f) = \conv \{ [\frac{1}{2},\frac{3}{2}]^{n-1} \times \{\frac{1}{2}\} \cup \{z,f\})$, where $z,f$ are the apexes of $S$ and the set $[\frac{1}{2},\frac{3}{2}]^{n-1} \times \{\frac{1}{2}\}$ is the base of $S$. Using the standard formula for the volume of double pyramids we obtain $\vol(S) = \frac{1}{n}$. Thus, $\vol(R_\Lambda(B,f) / \Lambda) = 2 = \det(\Lambda)$. It follows that $B$ has the unique-lifting property with respect to $\Lambda$.
\end{proof}

%We can use an invertible affine transformation to send the lattice $\Lambda$ from Proposition~\ref{prop:cube-const} to $\Z^n$, which would send the cube $B$ to a maximal lattice-free set which has the unique lifting property, but is not the representable as a coproduct by Proposition~\ref{prop:coproduct-nogo}. We remark that in the case $n=3$, the polytope $B$ discussed in the above proposition can be found in the three-dimensional enumeration result in \cite{MR2855866} (note that the authors in \cite{MR2855866} do not use arbitrary lattices but rather the integer lattice $\Z^n$).
We remark that in the case of dimension $n=3$, the polytope $B$ discussed in the above
proposition can be found in~\cite{MR2855866}, where the authors use the standard lattice
$\Z^3$ rather than $\Lambda$. The respective polytope is written as
$[o,e_1+ e_2] + [o,-e_1 + e_2] + [o,e_1 + e_2 + 2 e_3]$ rather than
$[0,2]^3$. To verify the equivalence with the example of the previous
proposition it suffices to check that the linear mapping sending
$e_1+e_2$ to $2 e_1$, $-e_1 + e_2$ to $2 e_2$ and $e_1 + e_2 + 2 e_3$ to $2 e_3$,
which maps bijectively the mentioned polytope from~\cite{MR2855866} to the
cube $[0,2]^3$, is also a bijection between $\Z^3$ and $\Lambda$. Thus, $B$ is a maximal lattice-free set which has the unique lifting property, but is not the representable as a coproduct by Proposition~\ref{prop:coproduct-nogo}.

\section{Characterization of special polytopes with the unique-lifting property}

\label{spec pyr sect}

\paragraph{Towards explicit description of polytopes with the unique-lifting property.} Providing an explicit description of all $n$-dimensional maximal lattice-free polytopes with the unique-lifting property for $n \ge 2$ is a challenging problem: so far, only the case $n=2$ has been settled completely. Already the case $n=3$ seems to be highly nontrivial. In the authors' opinion, one of the difficulties is that, in general, the set $\intr(R_F(B,f))$ with $f \in B \setminus F$ has complicated geometry whenever the relative interior of a facet $F$ of an $n$-dimensional maximal lattice-free polytope $B$ contains more than one integral point. It is thus interesting to analyze the somewhat more accessible special case in which $\relintr(F) \cap \Z^n$ consists of exactly one point for each facet $F$ of $B$. %Note that this case can be viewed as a generic situation, because the family of such special $n$-dimensional maximal lattice-free polytopes is dense in the family of all $n$-dimensional maximal lattice-free polytopes.
In this section we provide some partial information on the problem described above. 

We say that a closed set $S\subseteq \R^n$ with nonempty interior {\em translatively tiles $\R^n$} if $\R^n$ can be represented as $\bigcup_{u \in U} (S+u)$, where $U \subseteq \R^n$ and $\intr(S + u_1) \cap \intr(S + u_2) \ne \emptyset$ for all $u_1,u_2 \in U$ with $u_1 \ne u_2$. We say that $S$ tiles $\R^n$ by its integral translations if the latter condition holds with $U = \Z^n$. 

\begin{prop} \label{special situation thm}
	Let $n \in \N$ and let $B$ be a maximal lattice-free set in $\R^n$ such that the relative interior of each facet of $B$ contains exactly one integral point. Then $\vol(R(B,f)) = \vol(R(B,f) / \Z^n)$ and therefore, the following are equivalent: (a) $B$ has the unique-lifting property, (b) $\vol(R(B,f))=1$ for every $f \in B$ and (c) the topological closure of $\intr(R(B,f))$ tiles $\R^n$ by its integral translations for every $f \in B$.
	
%	the following assertions hold:
%	\begin{enumerate}[(a)]
%		\item One has $\vol(R(B,f)) \le 1$ for every $f \in B$.
%		\item If $\vol(R(B,f))=1$ for some $f \in \intr(B)$, then $B$ has the unique-lifting property. 
%		\item If $B$ has the unique-lifting property, then $\vol(R(B,f))=1$ for every $f \in B$, and, moreover, the topological closure of $\intr(R(B,f))$ tiles $\R^n$ by its integral translations.
%	\end{enumerate} 
\end{prop}
\begin{proof}
	Consider an arbitrary facet $F$ of $B$ and an arbitrary $f \in B$.
	By Lemma~\ref{when collision?}, no point of $\intr(R_F(B,f))$ is congruent to a point of another set $\intr(R_{F'}(B,f))$ modulo $\Z^n$, where $F'$ is a facet of $B$ with $F' \ne F$. Furthermore, taking into account the assumption on the facets of $B$, \eqref{struct congr:eq} implies that no two distinct points of $\intr(R(B,f))$ are congruent modulo $\Z^n$. Hence $\vol(R(B,f)) = \vol(R(B,f) / \Z^n) \le 1$. The latter implies the assertions (a), (b) and (c). \end{proof}

Thus, in our special situation, for verifying whether $B$ has the unique-lifting property, it suffices to compute the volume of the entire lifting region $R(B,f)$ rather than the set $R(B,f)/\Z^n$, which is a simplification. Nevertheless, since $R(B,f)$ is still quite a complicated set (e.g., not necessarily a convex one), checking $\vol(R(B,f))=1$ is not an easy task. 

\paragraph{Special pyramids with the unique-lifting property.}\label{sec:pyr-unique}

We analyze and partially characterize pyramids with the unique-lifting property. We shall use the following theorem, which is proved in the appendix.

\begin{theorem}\label{thm:trans-spindle} \thmheader{McMullen~\cite{mcmullen-email}}
Let $S\subseteq \R^n$ be an $n$-dimensional spindle that translatively tiles space.  Then $S$ is the image of the $n$-dimensional hypercube under an invertible affine transformation.
\end{theorem}

Theorem~\ref{thm:trans-spindle} can be used to prove the following.

\begin{theorem}\label{thm:base-condition} Let $n \in \N$ and let $B\subseteq \R^n$ be a maximal lattice-free polytope with the unique-lifting property such that $B$ is a pyramid whose base contains exactly one integral point in the relative interior. Then $B$ is a simplex.
\end{theorem}
\begin{proof}
Let $f$ be the apex and $F$ be the base of $B$ and let $z$ be the unique integral point in $\relintr(F)$. Since $B$ has the unique-lifting property, by Proposition~\ref{special situation thm}(c), the topological closure of $\intr(R(B,f))$ tiles $\R^n$ by its integral translations. The topological closure of $\intr(R(B,f))$ is $S= S_{F,z}(f)$, since $S_{F,z}(f)$ is the only full-dimensional spindle involved in the definition of $R(B,f)$. Thus, $S$ tiles $\R^n$ by translations. By Theorem~\ref{thm:trans-spindle}, $S$ is an image of a cube under an invertible affine transformation. In particular, the tangent cone at the apex $f$ is a simple cone. Therefore, $B$ is a simplex.
\end{proof}

The following theorem is proved in~\cite{bcm}.

\begin{theorem}\label{bcm-main}
Let $n \in \N$ and let $B$ be a maximal lattice-free simplex in $\R^n$ such that each facet of $B$ has exactly one integer point in its relative interior.  Then $B$ has the unique-lifting property if and only if $B$ is an affine unimodular transformation of $\conv(\{o, ne_1, \ldots, ne_n\})$.
\end{theorem}

We can now generalize this result to pyramids.

\begin{theorem}\label{thm:unique-integer}
Let $n \in \N$ and let $B$ be a maximal lattice-free pyramid in $\R^n$ such that every facet of $P$ contains exactly one integer point in its relative interior. Then $B$ has the unique-lifting property if and only if $B$ is an affine unimodular transformation of $\conv(\{o, ne_1, \ldots, ne_n\})$.
\end{theorem}
\begin{proof}
	Sufficiency follows from Corollary~\ref{cor:a1-an}.(a). For showing the necessity we assume that $B$ has the unique-lifting property. By Theorem~\ref{thm:base-condition}, $B$ is a simplex. Thus, Theorem~\ref{bcm-main} can be applied and the necessity follows immediately.
\end{proof}

\section*{Acknowledgements} 

We are grateful to Peter McMullen for providing an outline of the proof of Theorem~\ref{thm:cen-sym-spindle}, which is implemented in Section~\ref{sect spindles} of the appendix. We also would like to thank Martina Z\"ahle for pointing to \cite{MR2427002}.

\bibliographystyle{amsalpha}
\bibliography{literature}

\appendix
\section{Proofs of Propositions~\ref{thm:trivial-lifting-compute} and \ref{thm:trivial-lifting-minimal}}

\begin{proof}[Proof of Proposition~\ref{thm:trivial-lifting-compute}] 
First we show that the infimum in \eqref{psi star def} is attained. 

Consider the case of a bounded $B$. For a sufficiently  large $N \in \N$ the Euclidean ball of radius $N$ centered at $o$ contains $B$. It follows that $\phi_{B-f}(r) \ge \frac{1}{N} \|r\|$ for every $r \in \R^n$, where $\| \dotvar \|$ is the Euclidean norm. Consequently, $\phi_{B-f}(r+w) \ge \frac{1}{N} \| r + w \| \ge \frac{1}{N}  ( \|w\| - \|r\| ) > \phi_{B-f}(r)$ for $r \in \R^n$ and $w \in \Z^n$ whenever $w$ fulfills $\|w\| >  (N \phi_{B-f}(r) +\|r\|)$. It follows that $\inf_{w \in \Z^n} \phi_{B-f}(r + w)$ is attained for some of finitely many vectors $w \in \Z^n$ satisfying $\|w\| \le (N \phi_{B-f}(r) + \|r\|)$.

Let us switch to the case that $B$ is unbounded. It is known that the recession cone of $B$ is a linear space spanned by rational vectors; see see \cite{MR1114315}, \cite{MR2724071} and \cite{MR3027668}. Up to appropriate unimodular transformations, we can assume that $B$ has the form $B = B' \times \R^k$, where $k \in \{1,\ldots,n-1\}$ and $B'$ is a bounded maximal lattice-free set in $\R^{n-k}$. We denote by $\phi_{B'-f}$ the gauge-function of $B'$; it is well known that $\phi_{B-f}((r',r'')) = \phi_{B'-f}(r')$ for all $(r', r'') \in \R^{n-k}\times \R^{k}$. Thus, it suffices to apply the assertion of the bounded case to $B'$ to get the assertion for an unbounded $B$.

It remains to prove the assertion on polynomial-time computability. Assume that $n \in N$ is fixed and that $f$ and $a_i$ ($i \in I$) are rational vectors, whose components are given as the input in standard binary encoding. We have  
\begin{align} 
	\phi_{B-f}^\ast(r)  := &\min_{w \in \Z^n} \phi_{B-f} (r + w) \nonumber
	\\  = &\min_{w \in \Z^n} \max_{i\in I} a_i\cdot(r + w) \nonumber
	\\  = &\min \setcond{\rho \ge 0}{\rho \ge a_i(r+w) \ \forall i \in I, \ w \in \Z^n}. \label{special MILP}
\end{align}
Expression \eqref{special MILP} defines a mixed-integer linear program with rational coefficients with $n$ integer variables (the components of $w$) and one real variable (the value $\rho$).
Since $n$ is fixed, Lenstra's algorithm \cite{MR727410} can be used to determine \eqref{special MILP} in polynomial time.
\end{proof}

\begin{proof}[Proof of Proposition~\ref{thm:trivial-lifting-minimal}] %Let $\phi_{B-f}^*$ be the trivial lifting given by \eqref{psi star def}. 

It was established in~\cite{bcccz} that for every $r\in \R^n$ such that $r+ f \in R(f,B)$, $\pi(r) = \phi_{B-f}(r)$ for every minimal lifting $\pi$ of $\phi_{B-f}$. Moreover, it is not difficult to see that every minimal lifting is periodic with respect to $\Z^n$, i.e., $\pi(r) = \pi(r+w)$ for every $r\in \R^n$ and $w \in \Z^n$. If $\phi_{B-f}$ has a unique lifting, then $R(f,B) + \Z^n = \R^n$. Therefore, for any $r$, there exists $w \in \Z^n$ such that $r + w + f \in R(f,B)$ and thus $\pi(r) = \pi(r + w) = \phi_{B-f}(r+w) \geq \phi_{B-f}^*(r)$ for every minimal lifting $\pi$, thus establishing that $\phi_{B-f}^*$ is a minimal lifting.

Suppose $\phi_{B-f}$ does not have a unique minimal lifting. This implies there are at least two distinct minimal liftings and so there must exist a minimal lifting $\pi$ that is different from the lifting $\phi_{B-f}^*$. However, we show below that $\pi \leq \phi_{B-f}^*$. Thus, $\phi_{B-f}^*$ is not a minimal lifting. 

To show that $\pi \leq \phi_{B-f}^*$, consider any $r \in \R^n$. It is well-known that $\pi \leq \phi_{B-f}$ because $\pi$ is a minimal lifting. By Theorem~\ref{thm:trivial-lifting-compute}, there exists $w \in \Z^n$ such that $\phi_{B-f}^*(r) = \phi_{B-f}(r + w)$. By the $\Z^n$-periodicity of $\pi$, we have $\pi(r) = \pi(r+w) \leq \phi_{B-f}(r+w) = \phi_{B-f}^*(r)$.
\end{proof}

\section{Proof of Theorem~\ref{thm:trans-spindle}}

\label{sect spindles}

Let $P\subseteq \R^n$ be an $n$-dimensional centrally symmetric polytope with centrally symmetric facets. Let $G$ be any  $(n-2)$-dimensional face of $P$. The {\em belt} corresponding to $G$ is the set of all facets which contain a translate of $G$ or $-G$. Observe that every centrally symmetric polytope $P$ with centrally symmetric facets has belts of even size greater than or equal to $4$. 

A {\em zonotope} is a polytope given by a finite set of vectors $V=\{v_1, \ldots, v_k\} \subseteq \R^n$ in the following way: $Z(V) := \{\lambda_1v_1 + \ldots + \lambda_kv_k : -1 \leq \lambda_i \leq 1 \quad \forall i = 1, \ldots, k\}.$
We recall that $F(P,u)$ denotes the face of points in $P$ maximizing the linear function $x \mapsto \sprod{u}{x}$. The following simple lemma is well-known.
	
\begin{lemma} \label{lem:zonotope-belt}
	Let $n \in \N$. Let $V$ be a nonempty finite subset of $\R^n$ and let $u \in \R^n$. Then the face $F(Z(V),u)$ of the zonotope $Z(V)$ coincides, up to a translation, with the zonotope $Z\bigl( \setcond{v \in V}{\sprod{u}{v} = 0} \bigr)$. 
\end{lemma}
\begin{proof}
	By the Minkowski additivity of the functional $F(\dotvar,u)$, defined by \eqref{F in dir u eq}, we get $F(Z(V),u) = \sum_{v \in V} F([-v,v],u).$ 
%		\begin{equation}
%			\label{face of Z eq}
%			F(Z(V),u) = \sum_{v \in V} F([-v,v],u).
%		\end{equation}
		It is straightforward to verify that for every $v \in V$ one has %$F([-v,v],u) = \{-v\}$ if $\sprod{u}{v} < 0$, $F([-v,v],u) = [-v, v]$ if $\sprod{u}{v} = 0$ and $F([-v,v],u) = \{v\}$ if $\sprod{u}{v} > 0$. Putting these observations together, we have the assertion.
%		\begin{equation}
%			\label{summand of Z eq}
	\[		F([-v,v],u) := 
			\begin{cases}
				\{-v\} & \text{if} \ \sprod{u}{v} < 0, 
				\\ [-v,v] & \text{if} \ \sprod{u}{v} = 0, 
				\\ \{v\} & \text{if} \ \sprod{u}{v} > 0.
			\end{cases}
	\]
%		\end{equation}
Putting these observations together, we have the assertion.
		%Using \eqref{summand of Z eq} for the right hand side of \eqref{face of Z eq}, we get the assertion.
\end{proof}

The latter lemma shows that every face of a zonotope is a zonotope (and, thus, centrally symmetric). The following lemma deals with belts of zonotopes. Each belt of the cube $[-1,1]^n$ consists of exactly four facets. The following theorem shows that the latter property essentially characterizes cubes within all zonotopes.

%Given an $(n-2)$-dimensional face $G$ of an $n$-dimensional polytope $P$, the \emph{belt} of $P$ generated by $G$ is the set of all facets of $P$ that contain a translateion of $G$ or a translation of $-G$. 

\begin{theorem}
	Let $n \in \N$, $n \ge 3$. Let $V$ be a finite set linearly spanning $\R^n$ and such that each belt of the $n$-dimensional zonotope $Z(V)$ consists of exactly four facets. Then $Z(V)$ is the image of the $n$-dimensional cube $[-1,1]^n$ under a invertible linear transformation.
\end{theorem}
\begin{proof}
	Choose a basis $b_1,\ldots,b_n$ of $\R^n$ consisting of vectors in $V$. It suffices to show that every vector of $V$ is parallel to some vector of $\{b_1,\ldots,b_n\}$. After a change of coordinates in $\R^n$ we can assume that $b_1,\ldots,b_n$ is the standard basis $e_1,\ldots,e_n$. 
	
	Assume to the contrary, that there exists a vector $a = (\alpha_1,\alpha_2,\ldots,\alpha_n) \in V$ which is not parallel to any vector of the basis $e_1,\ldots,e_n$. Thus, at least two of its components $\alpha_1,\ldots,\alpha_n$ are nonzero.  Without loss of generality let $\alpha_1 \ne 0$ and $\alpha_2 \ne 0$. Let $W := V \cap (\{0\}^2 \times \R^{n-2})$. We have $e_3,\ldots,e_n \in W$ and $e_1,e_2, a \in V \setminus W$. Choose a nonzero vector $u' = \R^2 \times \{0\}^{n-2}$ such that $u'$ is not orthogonal to any vector from $V \setminus W$ (e.g., one can choose $u'= (1,\eps,0,\ldots,0)$, where $\eps>0$ is small). By Lemma~\ref{lem:zonotope-belt}, the face $G:=F(Z(V),u')$ is a translation of $Z(W)$. By the choice of $W$, the zonotope $Z(W)$ is $(n-2)$-dimensional.  We analyze the belt of $Z(V)$ determined by the $(n-2)$-dimensional face $G$.
	
	We shall construct a number of facets $F(Z(V),u)$ with $u \in \R^2 \times \{0\}^{n-2}$ belonging to the belt generated by $G$. In view of Lemma~\ref{lem:zonotope-belt}, for $u=e_1$ the face $F( Z(V), u)$ contains a translation of $Z( \{e_2\} \cup W)$. Similarly, for $u= e_2$ the face $F(Z(V),u)$ contains a translation of $Z(\{e_1\} \cup W)$. For a nonzero vector $u \in \R^2 \times \{0\}^{n-2}$ orthogonal to $a$ (say, for $u=(-\alpha_2,\alpha_1,0,\ldots,0)$) the face $F(Z(V),u)$ contains a translation of $Z(\{a\} \cup W)$. Since the zonotopes $Z(\{e_1\} \cup W)$, $Z(\{e_2\} \cup W)$ and $Z(\{a\} \cup W)$ are $(n-1)$-dimensional, we see that for all three choices of $u$ above, the face $F(Z(V),u)$ is actually a facet. The latter shows that the six distinct facets $F(Z(V),u)$ with $u \in \{\pm e_1,\pm e_2, \pm (-\alpha_2,\alpha_1,0,\ldots,0) \}$ belong to the belt generated by $G$. The latter is a  contradiction to the assumptions on $Z(V)$.
\end{proof}

\begin{theorem}\label{thm:cen-sym-spindle} \thmheader{McMullen~\cite{mcmullen-email}}
 Let $n \in \N$, $n \ge 3$, and let $S\subseteq \R^n$ be an $n$-dimensional spindle with centrally symmetric facets. Then $S$ is the image of the $n$-dimensional hypercube under an invertible affine transformation.
\end{theorem}
\begin{proof} 
	Since all facets of $S$ are centrally symmetric, by the Alexandrov-Shephard theorem (see \cite{MR0410566} for a short proof), the polytope $S$ itself is also centrally symmetric. Without loss of generality, we assume that $S$ is symmetric in the origin. Let $a$ and $-a$ be the apexes of the spindle $S$. 
	
We first show that every belt of $S$ is of length $4$. Let $G$ be an arbitrary $(n-2)$-dimensional face of $S$ and consider the belt of $S$ associated with $G$. Since $S$ is centrally symmetric, each belt is even length, i.e., of length $2 k$ where $k \ge 2$. There are $k$ facets $F_1, \ldots, F_{k }$ involved in this belt that contain $a$; the remaining $k$ facets contain $-a$. We project $S$ onto the two-dimensional space perpendicular to $G$ to get a polygon $P$.  The facets $F_1, \ldots, F_{k}$ are all projected onto $k$ distinct edges of the polygon $P$. Moreover, observe the projection of $a$ is contained in all these edges. Since $P$ is two-dimensional, intersection of more than three edges of $P$ is empty. Hence $k \le 2$ and since we also have $k \ge 2$, we get $k=2$.

We next show that all faces of $S$ are centrally symmetric. To do this, we first show that every $n-2$-dimensional face $G$ is centrally symmetric (for $n=3$ this is clear). For $i \in \{1,2\}$, by $c_i$ we denote the center of symmetry of $F_i$. Then $G$ has the form $F_i \cap F_j$ or $(-F_i) \cap (-F_j)$ or $F_i \cap (-F_j)$ with appropriate $i,j$ satisfying $\{i,j\} = \{1,2\}$. Consider the case $G = F_i \cap F_j = F_1 \cap F_2$. The symmetry of $F_1$ implies that $2 c_1 - G$ (the reflection of $G$ with respect to $c_1$) is a face of $F_1$. Since $a \in G$, the face $2 c_1 - G$ does not contain $a$. The face $G$ is contained in exactly two facets of $S$, both belonging to the belt $\{F_1,F_2,-F_1,-F_2\}$ generated by $G$. The facet $-F_1$ cannot contain $2 c_1 - G$, because $-F_1$ is opposite to $F_1$ and thus does not share any nonempty face with $F_1$. The facet $F_2$ of $S$ cannot contain $2 c_1 - G$, because $F_2$ contains $a$, while $2 c_1 - G$ does not contain $a$. It follows that the facet $-F_2$ contains $2 c_1 - G$. Then the reflection $-2c_2 - (2 c_1 - G)$ of $2c_1 -G$ with respect to the center $-c_2$ of $-F_2$ is a facet of $-F_2$. On the other hand, the reflection $-G$ of $G$ with respect to the center $o$ of $S$ is a face of $S$ which does not contain $a$. Hence $-G$ is a facet of $-F_2$. We have shown that $2 c_1 -G$, $-2 c_2 - 2 c_1 + G$ and $-G$ are facets of $-F_2$. Since all these facets of $F_2$ are parallel, two of them must coincide. We cannot have $2 c_1 - G = -G$, since this would imply $c_1 = o$ and, by this, $\relintr(F_1) \cap \intr(S) \ne \emptyset$, which is a contradiction. Consequently, $2 c_1 - G = -2 c_2 - 2 c_1 + G$ or $-2 c_2 - 2 c_1 + G = -G$, where each of the two equalities implies that $G$ is centrally symmetric. The case $G = (-F_i) \cap (-F_j)$ is completely analogous to the case $G = F_i \cap F_j$.

Let us switch to the case $G = F_i \cap (-F_j)$ with $\{i,j\} = \{1,2\}$. Without loss of generality, let $G = F_1 \cap (-F_2)$. The face $G$ of $S$ contains neither $a$ nor $-a$. The same also holds for the face $-G$ of $S$. The reflection $2 c_1 - G$ of $G$ with respect to the center $c_1$ of $F_1$ is a facet of $F_1$. Then $2 c_1 - G$ is not a facet of $-F_1$, because $2 c_1 - G$ is a facet of $F_1$, while $F_1$ and $-F_1$ are opposite facets of $S$. Thus, $2 c_1 - G$ is a facet of $F_2$ or $-F_2$. If $2c_1 - G$ is a facet of $F_2$, then also $2 c_2 - (2 c_1 - G)$ is a facet of $F_2$. It follows that $2 c_1 - G, 2 c_2 - 2c_1 + G$ and $-G$ are facets of $F_2$. Again, since they are all parallel, two of them must coincide. Coincidence of $2 c_1 - G$ and $-G$ implies $c_1 =o$ and yields a contradiction. Coincidence of any two other of these three facets of $F_2$ implies that $G$ is centrally symmetric. In the case that $2 c_1 - G$ is a facet of $-F_2$, we get that $- 2 c_2 - (2 c_1 - G)$ is a facet of $-F_2$. Thus, $G$, $-2 c_2- G$ and $- 2 c_2 - 2 c_1 + G$ are facets of $-F_2$. Coincidence of $G$ and $- 2 c_2 - 2 c_1 + G$ implies $c_1 = c_2$, yielding $\relintr(F_1) \cap \relintr(F_2) \ne \emptyset$, which is a contradiction. Coincidence of any other of these three facets of $S$ implies that $G$ is centrally symmetric.

It follows that every $(n-2)$-dimensional face of $S$ is centrally symmetric. Therefore, by a theorem of McMullen~\cite{mcmullen}, in the case $n \ge 4$, every face of $S$ is centrally symmetric (in the case $n=3$ this is clear from the assumptions). Consequently, all $2$-dimensional faces of $S$ are centrally symmetric and, by this, $S$ is a zonotope; see, for example,  \cite[Theorem~3.5.1]{MR1216521}. % or \cite[Proposition~2.2.14]{bjorner1999oriented}. 
Since $S$ is a zonotope whose belts are length $4$, by Lemma~\ref{lem:zonotope-belt}, $S$ is the image of the $n$-dimensional hypercube under an invertible affine transformation.\end{proof}

Theorem~\ref{thm:cen-sym-spindle} was communicated to us by Peter McMullen via personal email. We include a complete proof here as the result does not appear explicitly in the literature. The above proof is based on a proof sketch by Prof. McMullen.

We now state the celebrated Venkov-Alexandrov-McMullen theorem on translative tilings.

\begin{theorem}\label{thm:venkov} \thmheader{Venkov-Alexandrov-McMullen; see \cite[Theorem~32.2]{gruber}}
Let $P$ be a compact convex set with nonempty interior that translatively tiles $\R^n$. Then the following assertions hold: 
\begin{enumerate}[(a)]
	\item $P$ is a centrally symmetric polytope. 
	\item All facets of $P$ are centrally symmetric. 
	\item Every belt of $P$ is either length $4$ or $6$.
\end{enumerate}
\end{theorem}
\begin{proof}[Proof of Theorem~\ref{thm:trans-spindle}]
We only need to consider the case $n \ge 3$. The assertion follows directly from Theorem~\ref{thm:venkov} (assertions (a) and (b)) and Theorem~\ref{thm:cen-sym-spindle}.
\end{proof}

\end{document}